\documentclass[12pt]{amsart}
\usepackage[all]{xy}

\newcommand{\Cdb}{\mbox{$\mathbb{C}$}}

\newcommand{\Rdb}{\mbox{$\mathbb{R}$}}

\renewcommand{\H}{\mbox{${\mathcal H}$}}

\newcommand{\norm}[1]{\Vert#1\Vert}
\newcommand{\bignorm}[1]{\bigl\Vert#1\bigr\Vert}
\newcommand{\Bignorm}[1]{\Bigl\Vert#1\Bigr\Vert}

\newcommand{\cbnorm}[1]{\Vert#1\Vert_{cb}}
\newcommand{\bigcbnorm}[1]{\bigl\Vert#1\bigr\Vert_{cb}}

\newcommand{\rnorm}[1]{\Vert#1\Vert_{\rm reg}}
\newcommand{\minnorm}[1]{\Vert#1\Vert_{\rm min}}
\newcommand{\minten}{\otimes_{\rm min}}

\newcommand{\pten}{\hat{\otimes}}

\newtheorem{theorem}{Theorem}[section]
\newtheorem{lemma}[theorem]{Lemma}
\newtheorem{corollary}[theorem]{Corollary}
\newtheorem{proposition}[theorem]{Proposition}

\theoremstyle{remark}
\newtheorem{remark}[theorem]{\bf Remark}
\theoremstyle{definition}

\numberwithin{equation}{section}

\begin{document}
\baselineskip 15pt

\title[]{Dilations and rigid factorisations on noncommutative
$L^p$-spaces}

\author{Marius Junge and Christian Le Merdy}
\address{Mathematics Department\\ University of Illinois
\\ Urbana IL 61801\\ USA}
\email{junge@math.uiuc.edu}
\address{D\'epartement de Math\'ematiques\\ Universit\'e de  Franche-Comt\'e
\\ 25030 Besan\c con Cedex\\ France}
\email{lemerdy@math.univ-fcomte.fr}
\date{\today}

\begin{abstract} We study some factorisation and dilation properties of
completely positive maps on noncommutative $L^p$-spaces. We show
that Akcoglu's dilation theorem for positive contractions on
classical ($=$ commutative) $L^p$-spaces has no reasonable analog
in the noncommutative setting. Our study relies on non symmetric
analogs of Pisier's operator space valued noncommutative
$L^p$-spaces that we investigate in the first part of the paper.
\end{abstract}

\maketitle

\bigskip\noindent
{\it 2000 Mathematics Subject Classification : 46L07, 46L51,
46B28}

\bigskip

\section{Introduction.}
Akcoglu's dilation theorem \cite{A,AS} for positive contractions
on classical $L^p$-spaces plays a tremendous role in various areas
of analysis. The main result of this paper says that there is no
`reasonable' analog of that result for (completely) positive
contractions acting on noncommutative $L^p$-spaces. Recall that
Akcoglu's theorem essentially says that for any measure space
$(\Omega,\mu)$, for any $1<p<\infty$ and for any positive
contraction $u\colon L^p(\Omega)\to L^p(\Omega)$, there is another
measure space $(\Omega',\mu')$, two contractions $J\colon
L^p(\Omega)\to L^p(\Omega')$ and  $Q\colon L^p(\Omega')\to
L^p(\Omega)$, and an invertible isometry $U\colon L^p(\Omega')\to
L^p(\Omega')$ such that $u^n=QU^nJ$ for any integer $n\geq 0$. Let
$S^p$ be the $p$-th Schatten space of operators $a\colon
\ell^2\to\ell^2$ equipped with the norm $\norm{a}_p =(tr(\vert a
\vert^p))^{\frac{1}{p}}$. We show that if $p\neq 2$, there exists
a completely positive contraction $u\colon S^p\to S^p$ which is
not dilatable in the noncommutative sense. Namely whenever
$L^p(M)$ is a noncommutative $L^p$-space associated with a von
Neumann algebra $M$, there is no triple $(J,Q,U)$ consisting of
contractions $J\colon S^p\to L^p(M)$ and $Q\colon L^p(M)\to S^p$,
and of an invertible isometry $U\colon L^p(M)\to L^p(M)$, such
that $u^n=QU^nJ$ for any integer $n\geq 0$. Let $p'=p/(p-1)$ be
the conjugate number of $p$. We actually show the stronger result
that there is no pair $(T,S)$ of isometries $T\colon S^p\to
L^p(M)$ and $S\colon S^{p'}\to L^{p'}(M)$ such that $u=S^*T$.

The `need' of a noncommutative version of Akcoglu's theorem (and
its semigroup version \cite{F}) came out from some recent work of
Q. Xu and the authors devoted to diffusion semigroups on
noncommutative $L^p$-spaces \cite{JLX}. The lack of a
noncommutative Akcoglu's theorem turns out to be a key feature of
this topic.

We give two proofs of our main result. In Section 4, we give a non
constructive one, that is, we show the existence of a completely
positive contraction $u\colon S^p\to S^p$ which is not dilatable
without giving an explicit example. In Section 5, we provide a
second proof, which is longer but shows an explicit example. Our
proofs rely on various properties of a class of operator space
valued noncommutative $L^p$-spaces which we investigate in
Sections 2 and 3, and on $L^p$-matricially normed spaces
\cite{JLM}.

\bigskip
We will need a few techniques from operator space theory and we
refer the reader to either \cite{ER2} or \cite{P2} for the
necessary background on this topic. If $E,F$ are any two operator
spaces, we let $CB(E,F)$ denote the space of all completely
bounded maps $u\colon E\to F$. We let $\cbnorm{u}$ denote the
completely bounded norm of such a map and we say that $u$ is a
complete contraction if $\cbnorm{u}\leq 1$. We let $E\otimes_h F$
and $E\minten F$ denote the Haagerup tensor product and the
minimal tensor product of $E$ and $F$, respectively. Then we let
$\minnorm{\ }$ denote the norm on $E\minten F$.

For any integer $k\geq 1$ we let $M_k$ be the space of all
$k\times k$ matrices equipped with the operator norm and for any
$1\leq p<\infty$, we let $S^p_k$ be that space equipped with the
$p$-th Schatten norm. Also we use the notation $S^\infty$ for the
$C^*$-algebra of compact operators on $\ell^2$. Unless stated
otherwise, we let $(e_k)_{k\geq 1}$ denote the canonical basis of
$\ell^2$ and for any $i,j\geq 1$, we let
$E_{ij}\colon\ell^2\to\ell^2$ be the matrix unit taking $e_j$ to
$e_i$ and taking $e_k$ to $0$ for any $k\not= j$. If $X$ is any
vector space, we regard as usual $S^p_k\otimes X$ as the space of
all $k\times k$ matrices with entries in $X$, writing $[x_{ij}]$
for $\sum_{i,j} E_{ij}\otimes x_{ij}$ whenever $x_{ij}\in X$.

\medskip
\section{Some noncommutative operator space valued $L^p$-spaces.}
In this section we introduce a variant of the noncommutative
vector valued $L^p$-spaces considered by Pisier in \cite[Chapter
3]{P0} and we establish a few preliminary results. We refer the
reader to \cite{J} and \cite{JP} for related constructions. We
start with some background and preliminary results on
noncommutative $L^p$-spaces associated with a trace. We shall only
give a brief account on theses spaces and we refer to
\cite{T},\cite{FK}, \cite{PX2} and the references therein for more
details and further information.

\bigskip
We let $(M,\varphi)$ be a semifinite von Neumann algebra equipped
with a normal semifinite faithful trace $\varphi$. Then we let
\begin{equation}\label{4V}
V(M) =\bigcup eMe,
\end{equation}
where the union runs over all projections $e\in M$ such that
$\varphi(e)<\infty$. This is a $*$-algebra and the semifiniteness
of $\varphi$ ensures that $V(M)$ is $w^*$-dense in $M$. Let us
write $V=V(M)$ for simplicity and let $V_+=M_+\cap V$ denote the
positive part of $V$. Then any $a\in V_+$ has a finite trace.

Let $1\leq p<\infty$. For any $a\in V$, the operator $\vert
a\vert^p$ belongs to $V$ and we set
$$
\norm{a}_p\,=\,\bigl(\varphi(\vert
a\vert^p)\bigr)^{\frac{1}{p}},\qquad a\in V.
$$
Here $\vert a\vert = (a^*a)^{\frac{1}{2}}$ denotes the modulus of
$a$. It turns out that $\norm{\ }_p$ is a norm on $V$. By
definition, the noncommutative $L^p$-space associated with
$(M,\varphi)$ is the completion of $(V,\norm{\ }_p)$. It is
denoted by $L^p(M)$. For convenience, we also set
$L^{\infty}(M)=M$ equipped with the operator norm $\norm{\
}_\infty$.

Assume that $M\subset B(H)$ acts on some Hilbert space $H$, and
let $M'\subset B(H)$ denote the commutant of $M$. It will be
fruitful to have a description of the elements of $L^p(M)$ as
(possibly unbounded) operators on $H$. We say that a closed and
densely defined operator $a$ on $H$ is affiliated with $M$ if $a$
commutes with any unitary of $M'$. Then we say that an affiliated
operator $a$ is measurable (with respect to the trace $\varphi$)
provided that there is a positive real number $\lambda>0$ such
that $\varphi(\epsilon_\lambda)<\infty$, where $\epsilon_\lambda =
\chi_{[\lambda,\infty)}(\vert a\vert)$ is the projection
associated to the indicator function of $[\lambda,\infty)$ in the
Borel functional calculus of $\vert a \vert$. The set $L^0(M)$ of
all measurable operators is a $*$-algebra (see e.g. \cite[Chapter
I]{T} for a proof and a precise definition of the sum and product
on $L^0(M)$).

We recall further properties of $L^0(M)$ that will be used later
on. First for any $a$ in $L^0(M)$ and any $0<p<\infty$, the
operator $\vert a\vert^p =(a^*a)^{\frac{p}{2}}$ belongs to
$L^0(M)$. Second, let $L^0(M)_+$ be the positive part of $L^0(M)$,
that is, the set of all selfadjoint positive operators in
$L^0(M)$. Then the trace $\varphi$ extends to a positive tracial
functional on $L^0(M)_+$, still denoted by $\varphi$, in such a
way that for any $1\leq p<\infty$, we have
$$
L^p(M)\, =\, \bigl\{a\in L^0(M)\, :\, \varphi(\vert
a\vert^p)<\infty\bigr\},
$$
equipped with $\norm{a}_p = (\varphi(\vert a
\vert^p))^{\frac{1}{p}}$. Furthermore, $\varphi$ uniquely extends
to a bounded linear functional on $L^1(M)$, still denoted by
$\varphi$. For any $a,c\in L^0(M)$, we have $ac\in L^1(M)$ if and
only if $ca\in L^1(M)$ and in this case,
$\varphi(ac)=\varphi(ca)$. Furthermore we have
$$
\vert\varphi(a)\vert\leq\varphi(\vert a\vert)=\norm{a}_1
$$
for any $a\in L^1(M)$.

Let $1\leq p,q,s\leq \infty$ such that
$\frac{1}{p}+\frac{1}{q}=\frac{1}{s}$. The so-called
noncommutative H$\ddot{\rm o}$lder inequality asserts that
$L^p(M)\cdotp L^q(M) \subset L^s(M)$ and that we have
\begin{equation}\label{2Holder}
\norm{ac}_s \leq \norm{a}_p \norm{c}_q,\qquad a\in L^p(M),\ c\in
L^q(M).
\end{equation}
For any $1\leq p<\infty$, let $p'=p/(p-1)$ be the conjugate number
of $p$. Applying (\ref{2Holder}) with $q=p'$ and $s=1$, we may
define a duality pairing between $L^p(M)$ and $L^{p'}(M)$ by
\begin{equation}\label{2dual2}
\langle a,c \rangle\, =\, \varphi(ac), \qquad a\in L^p(M),\ c\in
L^{p'}(M).
\end{equation}
This induces an isometric isomorphism
$$
L^p(M)^*\,=\, L^{p'}(M),\qquad 1\leq p <\infty,\quad \frac{1}{p}
+\frac{1}{p'}=1.
$$
In particular, we may identify $L^1(M)$ with the (unique) predual
$M_*$ of $M$.

We will assume that the reader is familiar with complex
interpolation of Banach spaces, for which we refer to \cite{BL}.
We recall that by means of the embeddings of $L^\infty(M)$ and
$L^{1}(M)$ into $L^{0}(M)$, one may regard $(L^{\infty}(M),
L^{1}(M))$ as a compatible couple of Banach spaces and that we
have
\begin{equation}\label{4Inter}
[L^{\infty}(M),L^1(M)]_{\frac{1}{p}}\,=\, L^p(M),\qquad 1\leq p
\leq\infty,
\end{equation}
where $[\cdotp,\cdotp]_\theta$ denotes the complex interpolation
method.

\bigskip

For any $1\leq p<\infty$, we let $L^p(M)_+ = L^0(M)_+\cap L^p(M)$
denote the positive part of $L^p(M)$. We recall that the support
projection $Q$ of any element $b\in L^p(M)_+$ is the orthogonal
projection onto the closure of the range of $b$, and that ${\rm
ker}(Q)={\rm ker}(b)$. This projection belongs to $M$.

\begin{lemma}\label{4support}
Let $1\leq p,q,s\leq\infty$ such that $\frac{1}{p}+\frac{1}{q}
=\frac{1}{s}$ and $s<\infty$. Let $b\in L^p(M)_+$ and let $Q$ be
its support projection. Then $\overline{b L^q(M)}^{\norm{\ }_s} =
QL^s(M)$.
\end{lemma}

\begin{proof} Let $s'$ be the conjugate number of $s$.
Since $(QL^s(M))^{\perp} = L^{s'}(M)(1-Q)$, it suffices to show
that $(b L^q(M))^{\perp} = L^{s'}(M)(1-Q)$. If $c\in (b
L^q(M))^{\perp}$, then $\varphi(cba)=0$ for any $a\in L^q(M)$,
hence $cb=0$. This implies $cQ=0$, hence $c\in L^{s'}(M)(1-Q)$.
This proves one inclusion and the other one is obvious.
\end{proof}

\begin{lemma}\label{4Factor}
Assume that $p\geq 2$ and let $q\geq 2$ be defined by $\frac{1}{p}
+ \frac{1}{q}=\frac{1}{2}$. Let $y\in L^{p'}(M)$, $a\in L^q(M)_+$
and $b\in L^2(M)_+$ such that
$$
\bigl\vert\varphi(yzd)\bigr\vert\,\leq\,\norm{da}_2\norm{bz}_2
$$
for any $z\in M$ and $d\in L^p(M)$. Let $Q_a$ and $Q_b$ be the
support projections of $a$ and $b$, respectively. Then there
exists $w\in M$ such that $\norm{w}\leq 1$, $y=awb$ and
$w=Q_awQ_b$.
\end{lemma}

\begin{proof}
By Lemma \ref{4support}, $aL^p(M)$ and $bM$ are dense subspaces of
$Q_a L^2(M)$ and $Q_b L^2(M)$, respectively. Hence according to
our assumption, there exists a (necessarily unique) continuous
sesquilinear form $\sigma\colon Q_b L^2(M)\times Q_a L^2(M)\to
\Cdb$ such that $\sigma(bz, ad) = \varphi(yzd^*)$ for any $z\in M$
and any $d\in L^p(M)$. Let $\overline{\sigma}$ be the contractive
sesquilinear form on $L^2(M)$ defined by
$\overline{\sigma}(g,h)=\sigma(Q_bg,Q_ah)$ and let
$$
T\colon L^2(M)\longrightarrow L^2(M)
$$
be the associated linear contraction. By construction we have
\begin{equation}\label{4Ident}
\langle T(bz), ad\rangle_2 = \varphi(yzd^*),\qquad z\in M,\ d\in
L^p(M)
\end{equation}
and
\begin{equation}\label{4Identb}
\langle T(g),h\rangle_2 = \langle T(Q_bg),Q_ah\rangle_2,\qquad
g,h\in L^2(M),
\end{equation}
where $\langle\ ,\ \rangle_2$ denotes the inner product on
$L^2(M)$.

We claim that for any $c\in M$ and any $g\in L^2(M)$, we have
$T(gc) = T(g)c$. Indeed, for any $z\in M$ and $d\in L^p(M)$ we
have
$$
\langle T(bzc), ad\rangle_2 =\varphi(yzcd^*) = \varphi(yz(dc^*)^*)
=\langle T(bz), adc^*\rangle_2
$$
by (\ref{4Ident}). Consequently we have
$$
\langle T(Q_bgc),Q_ah\rangle_2 =\langle T(Q_bg),Q_ahc^*\rangle_2
$$
for any $g,h\in L^2(M)$, and hence
$$
\langle T(gc), h\rangle_2 =\langle T(g),hc^*\rangle_2 =\langle
T(g)c,h\rangle_2
$$
by (\ref{4Identb}). This proves the claim.

Consequently there exists $w\in M$, with $\norm{w}_\infty
=\norm{T}\leq 1$, such that $T(g)=wg$ for any $g\in L^2(M)$. Using
(\ref{4Ident}) again, we find  that
$$
\varphi(awb zd^*) = \varphi(w(bz)(ad)^*)=\varphi(yzd^*)
$$
for any $z\in M$ and any $d\in L^p(M)$. This shows that $y=awb$.

The identity  (\ref{4Identb}) ensures that $\langle Q_a wQ_b
g,h\rangle= \langle w g,h\rangle$ for any $g,h\in L^2(M)$. Hence
we have $w =Q_a w Q_b$.
\end{proof}

\bigskip
We introduce a notation which will be used throughout. Suppose
that $p,q,r,s\geq 1$ satisfy
$\frac{1}{q}+\frac{1}{r}+\frac{1}{s}=\frac{1}{p}$. Let $X$ be any
vector space, let $y\in L^r(M)\otimes X$ and let $(a_k)_k$ and
$(x_k)_k$ be finite families in $L^r(M)$ and $X$ respectively such
that $y=\sum_k a_k\otimes x_k$. Then for any $c\in L^q(M)$ and
$d\in L^s(M)$, we will write $cyd$ for the element of
$L^p(M)\otimes X$ defined by
$$
cyd\,=\,\sum_k ca_kd\otimes x_k.
$$

Let $F$ be an operator space, let $1\leq p<\infty$ and let $y\in
V\otimes F$. If $p\geq 2$, we let
$$
\norm{y}_{\alpha_p^\ell}\,=\,\inf\bigl\{\norm{c}_\infty
\minnorm{z}\norm{d}_{p}\bigr\},
$$
where the infimum runs over all $c,d\in V$ and all $z\in M\otimes
F$ such that $y=czd$. Here $\minnorm{z}$ denotes the norm of $z$
in $M\minten F$. Arguing as in the proof of \cite[Lemma 3.5]{P0},
it is not hard to check that $\norm{\ }_{\alpha_p^\ell}$ is a norm
on $V\otimes F$. The proof of the triangle inequality relies on
the convexity condition
$$
\bignorm{(d_1^* d_1 + d_2^*
d_2)^{\frac{1}{2}}}_p\leq\bigl(\norm{d_1}_p^2
+\norm{d_2}_p^2\bigr)^{\frac{1}{2}},\qquad d_1,d_2\in L^p(M),
$$
and the latter holds because $p\geq 2$.

If $p\leq 2$, we let $q\geq 2$ be such that
$\frac{1}{2}+\frac{1}{q} \,=\,\frac{1}{p}$, and we let
$$
\norm{y}_{\alpha_p^\ell}\,=
\,\inf\bigl\{\norm{a}_q\minnorm{z}\norm{b}_{2}\bigr\},
$$
where the infimum runs over all $a,b\in V$, and all $z\in M\otimes
F$ such that $y=azb$. Arguing again as in \cite[Lemma 3.5]{P0}, we
find that $\norm{\ }_{\alpha_p^\ell}$ is a norm on $V\otimes F$.
Then for any $p\geq 1$, we define the space
$$
L^p\{M;F\}_\ell
$$
as the completion of $V\otimes F$ for the norm $\norm{\
}_{\alpha_p^\ell}$.

Likewise, if $p\geq 2$, we let
$$
\norm{y}_{\alpha_p^r}\,=\,\inf\bigl\{\norm{c}_{p}\minnorm{z}\norm{d}_\infty
\bigr\},
$$
where the infimum runs over all $c,d\in V$ and all $z\in M\otimes
F$ such that $y=czd$. Then if $p\leq 2$ we let
$$
\norm{y}_{\alpha_p^r}\,=\,\inf
\bigl\{\norm{a}_2\minnorm{z}\norm{b}_{q}\bigr\},
$$
where the infimum runs over all $a,b\in V$, and all $z\in M\otimes
F$ such that $y=azb$. We obtain that $\norm{\ }_{\alpha_p^r}$ is a
norm on $V\otimes F$ as before, and we let
$$
L^p\{M;F\}_r
$$
be the completion of $V\otimes F$ for that norm.

In the case when $M=M_k$, these definitions reduce to the ones
given in \cite[Section 2]{JLM} and we have
$$
S^p_k\{F\}_\ell=L^p\{M_k;F\}_\ell\qquad\hbox{and}\qquad
S^p_k\{F\}_r=L^p\{M_k;F\}_r,
$$
where $S^p_k\{F\}_\ell$ and $S^p_k\{F\}_r$ are the spaces
introduced in the latter paper.

For any $\eta\in F^{*}$, the linear mapping $I_V\otimes\eta\colon
V\otimes F\to V$ (uniquely) extends to a bounded map
$\overline{\eta}\colon L^p\{M;F\}_\ell\to L^p(M)$, and we have
$\norm{\overline{\eta}}=\norm{\eta}$. Indeed assume for example
that $p\geq 2$, and let $y=czd\in V\otimes F$, with $c,d\in V$ and
$z\in M\otimes F$. Let $(a_k)_k$ and $(x_k)_k$ be finite families
in $M$ and $F$ respectively, such that $z=\sum_k a_k\otimes x_k$.
Then $(I_V\otimes\eta)y=\sum_k \langle\eta,x_k\rangle\,ca_k d$,
hence
\begin{align*}
\norm{(I_V\otimes\eta)y}_p\, &
\leq\,\norm{c}_\infty\Bignorm{\sum_k
\langle\eta,x_k\rangle\, a_k}_\infty\norm{d}_p\\
\, & \leq\,\norm{c}_\infty\norm{\eta}\minnorm{z}\norm{d}_p.
\end{align*}
Passing to the infimum over all $c,d,z$ factorising $y$, we obtain
that
$\norm{(I_V\otimes\eta)y}_p\leq\norm{\eta}\norm{y}_{\alpha_p^\ell}$.

Thanks to the above fact, we have a canonical (dense) inclusion
\begin{equation}\label{4Inc}
L^p(M)\otimes F\subset L^p\{M;F\}_\ell.
\end{equation}
More precisely, the bilinear mapping $V\times F\to V\otimes
F\subset L^p\{M;F\}_\ell$ obviously extends to a contractive
bilinear mapping $L^p(M)\times F\to L^p\{M;F\}_\ell$, which yields
a linear mapping $\kappa\colon L^p(M)\otimes F\to
L^p\{M;F\}_\ell$. Then we obtain (\ref{4Inc}) by showing that
$\kappa$ is one-to-one. For that purpose, let $y$ in
$L^p(M)\otimes F$ and assume that $\kappa(y)=0$. For any $\eta\in
F^*$, we have $(\overline{\eta}\circ\kappa) y = (I_{L^p}\otimes
\eta)y$, hence $(I_{L^p}\otimes \eta)y=0$. This shows that $y=0$.

The next lemma follows from the above discussion. We omit its easy
proof.

\begin{lemma}\label{4Pisier3}
\
\begin{itemize}
\item [(1)] Assume that $p\geq 2$. Then for any $z\in M\otimes F$
and any $d\in L^p(M)$, we have
$$
\norm{zd}_{L^p\{M;F\}_\ell}\,\leq\,\minnorm{z}\norm{d}_{p}.
$$
\item [(2)] Assume that $p\leq 2$, and that
$\frac{1}{2}+\frac{1}{q} \,=\,\frac{1}{p}$. Then for any $z\in
M\otimes F$ and any $a\in L^r(M),\, b\in L^2(M)$, we have
$$
\norm{azb}_{L^p\{M;F\}_\ell}\,\leq\,
\norm{a}_q\minnorm{z}\norm{b}_{2}.
$$
\item [(3)] The embedding (\ref{4Inc}) extends to a contractive
linear map $L^p(M)\pten F\to L^p\{M;F\}_\ell$, where $\pten$
denotes the  Banach space projective tensor product.
\end{itemize}
\end{lemma}

We end this section with an observation regarding opposite
structures. We recall that the opposite operator space of $F$,
denoted by $F^{\rm op}$, is defined as being the vector space $F$
equipped with the following matrix norms. For any $[x_{ij}]\in
M_k\otimes F$,
$$
\bignorm{[x_{ij}]}_{M_k(F^{\rm op})}\,=\,
\bignorm{[x_{ji}]}_{M_k(F)}.
$$
(See \cite[Section 2.10]{P2}.) Then $M^{\rm op}$ coincides with
the von Neumann algebra obtained by endowing $M$ with the reverse
product $*$ defined by $a*c=ca$ (for $a,c\in M$). It is clear from
the definition that $M\minten F = M^{\rm op}\minten F^{\rm op}$
isometrically. We deduce that we have an isometric identification
\begin{equation}\label{4Opposite}
L^p\{M;F\}_r\,\simeq\, L^p\{M^{\rm op};F^{\rm op}\}_\ell.
\end{equation}
Indeed assume for example that $p\geq 2$ and let $y\in V\otimes
F$. Suppose that the norm of $y$ in $L^p\{M;F\}_r$ is $<1$. Then
there exist $c,d\in V$ and $z\in M\otimes F$ such that $y=czd$,
$\norm{c}_p<1$, $\norm{d}_\infty <1$  and $\norm{z}_{M\minten
F}<1$. Let us write $z=\sum_k a_k\otimes x_k,$ with $a_k\in M$ and
$x_k\in F$, so that $y=\sum_k ca_kd\otimes x_k$. Then $ca_kd=d*
a_k* c$ for any $k$, hence $y=d*\bigl(\sum_k a_k\otimes x_k\bigr)*
c= d*z*c$. Since $\norm{z}_{M\minten F}= \norm{z}_{M^{\rm
op}\minten F^{\rm op}}$, this implies that the norm of $y$ in
$L^p\{M^{\rm op};F^{\rm op}\}_\ell$ is $<1$. Reversing the
argument we find that the norms of $y$ in $L^p\{M;F\}_r$ and in
$L^p\{M^{\rm op};F^{\rm op}\}_\ell$ actually coincide.

\medskip
\section{Duality for $L^p\{M;F\}_\ell$.}
We let $R$ and $C$ be the standard row and column Hilbert spaces,
and we denote by $R_k$ and  $C_k$ their $k$-dimensional versions,
respectively. This section is devoted to various properties of the
dual space of $L^p\{M;F\}_\ell$, especially when $F=R$. We will
start with a description of the dual space of $S^p_k\{F\}_\ell$
for any $F$.

We recall that if $E_0$ and $E_1$ are any two operator spaces, and
if $(E_0, E_1)$ is a compatible couple in the sense of Banach
space interpolation theory, then $[E_0,E_1]_{\theta}$ has a
canonical operator space structure. Indeed its matrix norms are
given by the isometric identities
$M_k\bigl([E_0,E_1]_{\theta}\bigr)\,=\,\bigl[M_k(E_0),M_k(E_1)
\bigr]_{\theta}$. See \cite[Section 2.7]{P2} and \cite{P5} for
details and complements. For any $\theta\in[0,1]$, we let
$$
R(\theta)\,=\,[R,C]_\theta
$$
be the Hilbertian operator space obtained by applying this
construction to the couple $(R,C)$. Then we both have
$$
R(\theta)^*\,=\, R(1-\theta)\qquad\hbox{and}\qquad R(\theta)^{\rm
op} \,=\, R(1-\theta)
$$
completely isometrically for any $\theta\in [0,1]$.

Let $F$ be an operator space. We may identify $S^p_k\otimes F$
with $\ell^2_k\otimes F\otimes \ell^2_k$ be identifying
$e_i\otimes x\otimes e_j$ with $E_{ij}\otimes x$ for any $x\in F$
and any $1\leq i,j\leq k$. According to \cite{JLM}, this induces
isometric identifications
\begin{equation}\label{2Pisier5}
S^p_k\{F\}_\ell\,\simeq\, C_k \otimes_h F\otimes_h
R_k\bigl(\tfrac{2}{p}\bigr)\quad\hbox{and}\quad
S^p_k\{F\}_r\,\simeq\, R_k\bigl(1-\tfrac{2}{p}\bigr) \otimes_h
F\otimes_h R_k
\end{equation}
if $p\geq 2$, whereas
\begin{equation}\label{2Pisier6}
S^p_k\{F\}_\ell\,\simeq\,
R_k\bigl(2\bigl(1-\tfrac{1}{p}\bigr)\bigr) \otimes_h F\otimes_h
C_k\quad\hbox{and}\quad S^p_k\{F\}_r\,\simeq\, R_k \otimes_h
F\otimes_h R_k\bigl(\tfrac{2}{p}-1\bigr)
\end{equation}
if $p\leq 2$.

\begin{proposition}\label{2Duality2}
Let $1<p,p'<\infty$ be conjugate numbers and let $F$ be an
operator space. Then we have isometric identifications
\begin{equation}\label{2Pisier7}
\bigl(S^p_k\{F\}_\ell\bigr)^*\,\simeq\,S^{p'}_k\{F^{*\,\rm
op}\}_\ell \qquad\hbox{ and }\qquad
\bigl(S^p_k\{F\}_r\bigr)^*\,\simeq\,S^{p'}_k\{F^{*\,\rm op}\}_r
\end{equation}
through the duality pairing $(S^p_k\otimes
F)\times(S^{p'}_k\otimes F^*)\to\Cdb$ mapping the pair $(a\otimes
x,c\otimes\eta)$ to the complex number $tr(ac)
\langle\eta,x\rangle$ for any $a\in S^p_k$, $c\in S^{p'}_k$, $x\in
F$ and $\eta\in F^*$.
\end{proposition}

\begin{proof}
We will use the fact that if $E_1,\ldots, E_n$ are any operator
spaces, then $E_1\otimes_h\cdots\otimes_h E_n$ is isometrically
isomorphic to $E_n^{\rm op}\otimes_h \cdots\otimes_hE_1^{\rm op}$
via the linear mapping taking $x_1\otimes\cdots\otimes x_n$ to
$x_n\otimes\cdots\otimes x_1$ for any $x_1\in E_1, \ldots, x_n\in
E_n$ (see e.g. \cite[p. 97]{P2}).

We only prove the first identity in (\ref{2Pisier7}), the second
one being similar. We use the self-duality of the Haagerup tensor
product (see e.g. \cite[Th. 9.4.7]{ER2}). Assume that $p\geq 2$.
By the above observations, we have
\begin{align*}
\bigl(S^p_k\{F\}_\ell\bigr)^*\,& \simeq\, C_k^*\otimes_h
F^*\otimes_h R_k\bigl(\tfrac{2}{p}\bigr)^*\\
& \simeq\, R_k \otimes_h F^*\otimes_h
R_k\bigl(1-\tfrac{2}{p}\bigr)\\
& \simeq\, R_k\bigl(1-\tfrac{2}{p}\bigr)^{\rm op}\otimes_h F^{*\,
\rm op}\otimes_h
R_k^{\rm op}\\
& \simeq\, R_k\bigl(\tfrac{2}{p}\bigr)\otimes_h F^{*\, \rm
op}\otimes_h C_k\\
& \simeq\, S^{p'}_k\{F^{*\,\rm op}\}_\ell.
\end{align*}
Moreover it is not hard to check (left to the reader) that the
duality pairing leading to these isometric isomorphisms is the one
given in the statement.

The proof for $p\leq 2$ is similar.
\end{proof}

\begin{remark}\label{2Duality4}
Let $S^p_k[F]$ denote Pisier's operator space valued Schatten
space \cite[Chapter 1]{P0}. We recall that for any $y\in
S^p_k\otimes F$, the norm $\norm{y}_{S^p_k[F]}$  is equal to
$\inf\{\norm{c}_{2p}\minnorm{z}\norm{d}_{2p}\}$, where the infimum
runs over all $c,d\in S^{2p}_{k}$ and all $z\in M_k(F)=M_k\minten
F$ such that $y=czd$. Moreover we have
\begin{equation}\label{2Pisier8}
S^p_k[F]\simeq R_k\bigl(1-\tfrac{1}{p}\bigr)\otimes_h F\otimes_h
R_k\bigl(\tfrac{1}{p}\bigr)
\end{equation}
isometrically. Then the proof of Proposition \ref{2Duality2}
yields an isometric identification
\begin{equation}\label{2Duality3}
S^p_k[F]^*\simeq\,S^{p'}_k[F^{*\rm op}].
\end{equation}
Using transposition, the latter result is the same as \cite[Cor.
1.8]{P0}.

We finally observe that in general the identifications in
(\ref{2Pisier7}) are not completely isometric (already with
$k=1$).
\end{remark}

\bigskip
The above proposition leads to a natural  duality problem, which
turns out to be crucial for our investigations in the next two
sections. Let $1<p,p'<\infty$ be two conjugate numbers, and
consider an arbitrary semifinite von Neumann algebra
$(M,\varphi)$. For any operator space $F$, consider the duality
pairing
$$
(L^p(M) \otimes F)\times(L^{p'}(M) \otimes F^*)\longrightarrow
\Cdb
$$
defined by
\begin{equation}\label{4Duality1}
(a\otimes x,c\otimes\eta)\,\longmapsto\,
\varphi(ac)\,\langle\eta,x\rangle
\end{equation}
for any $a\in L^p(M)$, $c\in L^{p'}(M)$, $x\in F$  and  $\eta\in
F^*$. In view of Proposition \ref{2Duality2}, it is natural to
wonder whether this pairing induces an isometric embedding of
$L^{p'}\{M;F^{*\rm op}\}_\ell$ into $L^p\{M;F\}_\ell^*$. Arguing
as in the proof of \cite[Th. 4.1]{P0}, and using Proposition
\ref{2Duality2}, we may obtain that this holds true when $M$ is
hyperfinite. However it is false in general, see Remark \ref{4Rem}
(2) below. In the rest of this section we will focus on the
special case when $F=R$ and we will show a positive result in that
case.

We recall that $R^*=C$ and that $C^{\rm op}=R$, so that  $R^{*\rm
op}=R$. In Sections 4 and 5, we will use the fact that for any
$1<p<\infty$, the above pairing induces a contraction
$L^{p'}\{M;R\}_\ell\to  L^{p}\{M;R\}_\ell^*$. The next theorem is
a more precise result that we prove for the sake of completeness.

\begin{theorem}\label{4Duality}
\
\begin{itemize}
\item [(1)] For any $1<p\leq 2$, we have
$$
L^{p'}\{M;R\}_\ell\,\hookrightarrow\, L^{p}\{M;R\}_\ell^*
\qquad\hbox{isometrically}.
$$
\item [(2)] For any $2<p<\infty$, we have an isometric isomorphism
$$
L^p\{M;R\}_{\ell}^*\,\simeq\, L^{p'}\{M;R\}_\ell.
$$
\end{itemize}
\end{theorem}

\bigskip
In the sequel we let $(e_n)_{n\geq 1}$ denote the canonical basis
of $R$ and we recall that for any finite sequence $(z_n)_n$ in
$M$, we have
$$
\Bignorm{\sum_n z_n\otimes e_n}_{M\minten R}\,=\,\Bignorm{\sum_n
z_nz_n^*}_\infty^{\frac{1}{2}}.
$$

\begin{lemma}\label{4Convexity}
Let $2\leq p<\infty$. For any finite families $(d_j)_j$ in
$L^p(M)$  and $(z_{nj})_{n,j}$ in $M$, we have
$$
\Bignorm{\sum_{n,j} z_{nj}d_j\otimes
e_n}_{L^p\{M;R\}_\ell}\,\leq\, \Bignorm{\Bigl(\sum_j
d_j^{*}d_j\Bigr)^{\frac{1}{2}}}_p\, \Bignorm{\sum_{n,j}
z_{nj}z_{nj}^*}_\infty^{\frac{1}{2}}.
$$
\end{lemma}

\begin{proof} We suppose that $M\subset B(H)$ as before. Let
$d=\Bigl(\sum_j d_j^{*}d_j\Bigr)^{\frac{1}{2}}$ and let $Q$ be its
support projection. For any $j$, we have $0\leq d_j^{*}d_j\leq
d^{2}$ hence there exists a (necessarily unique) $w_j\in M$ such
that
$$
w_jd=d_j\qquad\hbox{and}\qquad w_jQ=w_j.
$$
Then we have
$$
d^{2}\,=\,\sum_j d_j^{*}d_j\,=\,d\Bigl(\sum_j
w_j^{*}w_j\Bigr)d\qquad\hbox{and}\qquad Q\Bigl(\sum_j
w_j^{*}w_j\Bigr)Q\,=\, \sum_j w_j^{*}w_j.
$$
This readily implies that $\sum_j w_j^{*}w_j =Q$. Indeed, these
two bounded operators coincide on the range of $d$ and on the
kernel of $Q$. In particular, we have
$$
\Bignorm{\sum_j w_j^{*}w_j}_\infty\leq 1.
$$

Let $g_1,\ldots,g_n,\ldots$ and $h$ be elements of $H$. Then
$$
\sum_n\Bigl\langle\Bigl(\sum_j
z_{nj}w_j\Bigr)g_n,h\Bigr\rangle\,=\, \sum_{n,j}\bigl\langle
w_j(g_n), z_{nj}^{*}(h)\bigr\rangle.
$$
Hence by Cauchy-Schwarz, we have
\begin{align*}
\Bigl\vert \sum_n\Bigl\langle\Bigl(\sum_j
z_{nj}w_j\Bigr)g_n,h\Bigr\rangle \Bigr\vert\, &\leq\,
\Bigl(\sum_{n,j}\norm{w_j(g_n)}^{2}\Bigr)^{\frac{1}{2}}\,
\Bigl(\sum_{n,j}\norm{z_{nj}^{*}(h)}^{2}\Bigr)^{\frac{1}{2}}\\
&\leq\,\Bignorm{\sum_j w_j^{*}w_j}_\infty^{\frac{1}{2}}\,
\Bigl(\sum_{n}\norm{g_n}^{2}\Bigr)^{\frac{1}{2}}\,
\Bignorm{\sum_{n,j} z_{nj} z_{nj}^{*}}_\infty^{\frac{1}{2}}\,\norm{h}\\
&\leq\, \Bignorm{\sum_{n,j} z_{nj}
z_{nj}^{*}}_\infty^{\frac{1}{2}}\,
\Bigl(\sum_{n}\norm{g_n}^{2}\Bigr)^{\frac{1}{2}}\,\norm{h}.
\end{align*}
For any $n\geq 1$, let
$$
z'_n\,=\,\sum_j z_{nj} w_j.
$$
The above calculation shows that
$$
\Bignorm{\sum_n z'_n\otimes e_n}_{M\minten R}\,\leq\,
\Bignorm{\sum_{n,j} z_{nj} z_{nj}^{*}}_\infty^{\frac{1}{2}}.
$$
Moreover we have
$$
\Bignorm{\sum_{n,j} z_{nj}d_j\otimes e_n}_{L^p\{M;R\}_\ell}\,=\,
\Bignorm{\Bigl(\sum_{n} z'_{n}\otimes
e_n\Bigr)d}_{L^p\{M;R\}_\ell} \,\leq\,\norm{d}_p\,\Bignorm{\sum_n
z'_n\otimes e_n}_{M\minten R}
$$
by Lemma \ref{4Pisier3} (1). The result follows at once.
\end{proof}

\begin{proof} (Of Theorem \ref{4Duality}.)
The first step of the proof will consist in showing that for any
$2\leq p<\infty$, we have
\begin{equation}\label{4Duality3}
L^{p'}\{M;R\}_\ell\,\subset\, L^{p}\{M;R\}_\ell^*
\qquad\hbox{isometrically}.
\end{equation}
We let $V=V(M)$ be given by (\ref{4V}) and we let $\H\subset R$ be
the linear span of the $e_n$'s. By Lemma \ref{4Pisier3} (3),
$V\otimes \H$ is both dense in $L^p\{M;R\}_\ell$ and
$L^{p'}\{M;R\}_\ell$. In the sequel we regard $V\otimes \H$ as the
space of finite sequences in $V$. Indeed we identify such a
sequence $(y_n)_n$ with $\sum_{n\geq 1} y_n\otimes e_n$.

We let $q\geq 2$ such that $\frac{1}{2}+\frac{1}{q}=\frac{1}{p'}$.
Equivalently,
$$
\frac{1}{q} +\frac{1}{p}=\frac{1}{2}\,.
$$
Let $y=(y_n)_n$ and $y'=(y_n')_n$ in $V\otimes \H$. Let $c,d\in V$
and let $(z_n)_n$ be a sequence of $M$ such that $y_n=cz_nd$ for
any $n\geq 1$. Likewise, let $a,b\in V$ and let $(z_n')_n$ be a
sequence of $M$ such that $y_n'=az_n'b$ for any $n\geq 1$. The
duality pairing $\langle y,y'\rangle$ from (\ref{4Duality1}) is
given by
$$
\langle y,y'\rangle\,=\,\sum_{n}\varphi(y_n
y'_n)\,=\,\sum_{n}\varphi(cz_nda z_n'b)\,=\,
\sum_{n}\varphi(bcz_nda z_n').
$$
By Cauchy-Schwarz, we deduce that
\begin{align*}
\bigl\vert \langle y,y'\rangle\bigr\vert \, & \leq\, \sum_{n}
\bigl\vert\varphi(bcz_ndaz_n')\bigr\vert\,\leq \sum_{n}
\norm{bcz_n}_2\norm{daz_n'}_2\\
& \leq\, \Bigl(\sum_{n} \norm{bcz_n}_2^2\Bigr)^{\frac{1}{2}}\,
\Bigl(\sum_{n} \norm{daz_n'}_2^2\Bigr)^{\frac{1}{2}}.
\end{align*}
Moreover we have
$$
\sum_{n}\norm{bcz_n}_2^2\,= \sum_{n}
\varphi(bcz_nz_n^*c^*b^*)\,=\varphi\Bigl(bc\Bigl( \sum_n
z_nz_n^*\Bigr)c^*b^*\Bigr)\,\leq\norm{c}_\infty^2\norm{b}_2^2
\,\Bignorm{ \sum_n z_nz_n^*}_\infty.
$$
Likewise,
$$
\sum_{n}\norm{daz_n'}_2^2\,\leq\norm{da}_2^2\,\Bignorm{ \sum_n
z_n'{z_n'}^*}_\infty,
$$
and hence
$$
\sum_{n\geq
1}\norm{daz_n'}_2^2\,\leq\norm{d}_p^2\norm{a}_q^2\,\Bignorm{
\sum_n z_n'{z_n'}^*}_\infty.
$$
All together we deduce that
$$
\bigl\vert \langle y ,y'\rangle\bigr\vert \,\leq\, \norm{d}_p
\norm{a}_q \norm{c}_\infty\norm{b}_2\, \Bignorm{\sum_n z_n\otimes
e_n}_{M\minten R}\Bignorm{\sum_n z_n'\otimes e_n}_{M\minten R}.
$$
Passing to the infimum over all possible $a,b,c,d\in V$ and $z_n,
z_n'$ in $M$ as above, we deduce that
$$
\bigl\vert \langle y,y'\rangle\bigr\vert
\,\leq\,\norm{y}_{L^p\{M;R\}_\ell}\,
\norm{y'}_{L^{p'}\{M;R\}_\ell}.
$$
This shows that the duality pairing (\ref{4Duality1}) for $F=R$
induces a contraction
$$
L^{p'}\{M;R\}_\ell\longrightarrow L^p\{M;R\}_\ell^*.
$$

To show that this contraction is actually an isometry, we let
$y'=(y'_n)_n$ in $V\otimes \H$, we let $\zeta\colon
L^p\{M;R\}_{\ell}\to\Cdb$ be the corresponding functional and we
assume that $\norm{\zeta}\leq 1$. According to Lemma
\ref{4Convexity} we have
\begin{align*}
\Bigl\vert\sum_{n,j} \varphi (y'_n z_{nj} d_j)\,\Bigr\vert\,
&=\,\Bigl\vert \Bigl\langle \zeta,  \sum_{n,j} z_{n j}d_j\otimes
e_n\Bigr\rangle\Bigr\vert\\
& \leq\, \Bignorm{\Bigl(\sum_j d_j^{*}d_j\Bigr)^{\frac{1}{2}}}_p\,
\Bignorm{\sum_{n,j} z_{nj}z_{nj}^*}_\infty^{\frac{1}{2}}
\end{align*}
for any finite families $(d_j)_j$ in $L^p(M)$ and $(z_{nj})_{n,j}$
in $M$. Multiplying each $z_{nj}$ by an appropriate complex number
of modulus one, we deduce that
$$
\sum_{n,j} \bigl\vert \varphi(y'_n  z_{nj} d_j) \bigr\vert\, \leq
\, \Bignorm{\Bigl(\sum_j d_j^{*}d_j\Bigr)^{\frac{1}{2}}}_p\,
\Bignorm{\sum_{n,j} z_{nj}z_{nj}^*}_\infty^{\frac{1}{2}}.
$$
Note that $\frac{q}{2}$  is the conjugate number of $\frac{p}{2}$
and let $K_1$ be the positive part of the unit ball of
$L^\frac{q}{2}(M)$, equipped with the $\sigma(L^\frac{q}{2}(M),
L^{\frac{p}{2}}(M))$-topology. Likewise, let $K_2$ be the positive
part of the unit ball of $M^*$, equipped with the $w^*$-topology.
Since $\norm{d}_p^2=\norm{d^*d}_{\frac{p}{2}}$ for any $d\in
L^p(M)$, it follows from above that for any $(d_j)_j$ in $L^p(M)$
and any $(z_{nj})_{n,j}$ in $M$, we have
$$
2\,\sum_{n,j} \bigl\vert \varphi(y'_n  z_{nj} d_j) \bigr\vert\,
\leq \, \sup_{A\in K_1} \varphi\Bigl(\Bigl(\sum_j d_j^*d_j\Bigr)
A\Bigr) \,+\,\sup_{B\in K_2}\Bigl\langle B,\sum_{n,j}
z_{nj}z_{nj}^*\Bigr\rangle.
$$
Since $K_1$ and $K_2$ are compact, we deduce from \cite[Lemma
2.3.1]{ER2} (minimax principle) that there exist $A\in K_1$ and
$B\in K_2$ such that
$$
2\,\sum_{n,j} \bigl\vert \varphi(y'_n  z_{nj} d_j) \bigr\vert\,
\leq \, \varphi\Bigl(\Bigl(\sum_j d_j^*d_j\Bigr) A\Bigr) \,+\,
\Bigl\langle B,\sum_{n,j} z_{nj}z_{nj}^*\Bigr\rangle
$$
for any $d_j$ and $z_{nj}$ as above. Using the classical identity
$2st=\inf_{\delta>0} \delta t^2 + \delta^{-1}s^2$ for nonnegative
real numbers $s,t\geq 0$, we finally deduce that
\begin{equation}\label{4Ineq}
\sum_{n} \bigl\vert \varphi(y'_n  z_n d)\bigr\vert\, \leq
\,\varphi( d^*d A)^{\frac{1}{2}}\,\Bigl\langle B,\sum_n z_n
z_n^*\Bigr\rangle^{\frac{1}{2}},\qquad d\in L^p(M),\ z_n\in M.
\end{equation}
We now argue as in the proof of  \cite[Prop. 2.3]{H} to show that
$B$ may be replaced by its normal part in the above estimate. Let
$B_{\rm sing}$ be the singular part of $B$. It is shown in
\cite{H} that there is an increasing net $(e_t)_t$ of projections
in $M$ converging to $1$ in the $w^{*}$-topology, such that
$B_{\rm sing}(e_t)=0$ for any $t$. This implies that
$$
\Bigl\langle B_{\rm sing},\sum_n (e_tz_n) (e_tz_n)^*\Bigr\rangle\,
=\, \Bigl\langle  B_{\rm sing} ,e_t\Bigl(\sum_n z_n z_n^*\,\Bigr)
e_t\Bigr\rangle\, =0.
$$
Since $\varphi(y'_n  z_n d)=\lim_t\varphi(y'_n  e_tz_n d)$, this
implies that (\ref{4Ineq}) holds true with $B-B_{\rm sing}$
instead of $B$.

Thus we may assume that $B$ is normal, and we regard it as an
element of $L^1(M)_+$. Let $b=B^{\frac{1}{2}}\in L^2(M)_+$ be its
square root. For any $z_1,\ldots, z_n,\ldots$ in $M$, we have
$$
\Bigl\langle B,\sum_n z_nz_n^*\Bigr\rangle = \sum_n \varphi( b^2
z_nz_n^*)=\sum_n\norm{bz_n}_2^2.
$$
Likewise if we let $a=A^{\frac{1}{2}}\in L^q(M)_+$, then we have
$\varphi( d^*d A)=\norm{da}_2^2$ for any $d\in L^p(M)$.
Consequently, we have
$$
\sum_{n} \bigl\vert \varphi(y'_n z_n d) \bigr\vert\, \leq
\,\norm{da}_2\,\Bigl(\sum_n \norm{bz_n}_2^{2}
\Bigr)^{\frac{1}{2}},\qquad d\in L^p(M),\ z_n\in M.
$$
Applying Lemma \ref{4Factor} to each $y'_n$, we deduce that there
is a finite sequence $(w_n)_n$ in $M$ such that $y'_n=aw_nb$ and
$w_n=Q_aw_nQ_b$ for any $n\geq 1$, where $Q_a$ and $Q_b$ denote
the support projections of $a$ and $b$, respectively. Since
$L^p(M) a$ is dense in $L^2(M)Q_a$, and $b M$ is dense in $Q_b
L^2(M)$ (see Lemma \ref{4support}), the above estimate yields
$$
\Bigl\vert\varphi\Bigl(\sum_{n} w_n g_n
h\Bigr)\Bigr\vert\,\leq\,\norm{h}_2\, \Bigl(\sum_{n}
\norm{g_n}_2^2\Bigl)^{\frac{1}{2}},\qquad h\in L^2(M)Q_a,\ g_n\in
Q_b L^2(M).
$$
Since  $w_n=Q_aw_nQ_b$ this implies that
$$
\Bigl\vert\varphi\Bigl(\sum_{n} w_n g_n
h\Bigr)\Bigr\vert\,\leq\,\norm{h}_2\,
\Bigl(\sum_{n}\norm{g_n}_2^2\Bigl)^{\frac{1}{2}},\qquad h\in
L^2(M),\ g_n\in L^2(M).
$$
Regarding $M\subset B(L^2(M))$ in the usual way, we deduce that
$\bignorm{\sum_{n}w_n w_n^*}_{\infty}\leq 1$. Appealing to Lemma
\ref{4Pisier3} (2), this proves that
$\norm{y'}_{L^{p'}\{M;R\}_\ell}\leq 1$, and concludes the proof of
(\ref{4Duality3}).

The latter intermediate result implies that for any $2\leq
p<\infty$, we have
\begin{equation}\label{4Duality6}
L^p\{M;R_N\}_{\ell}^*\,\simeq\, L^{p'}\{M;R_N\}_\ell
\end{equation}
for any integer $N\geq 1$. Since the above spaces are reflexive,
this implies that (\ref{4Duality6}) actually holds true for any
$1<p<\infty$. In turn this implies that (\ref{4Duality3}) holds
true for any $1<p<\infty$, because $V\otimes\H$ is dense in
$L^{p'}\{M;R\}_\ell$. In particular we obtain part (1) of the
theorem.

\bigskip
We now turn to the proof of (2), which will consist in showing
that for $2<p<\infty$, the isometry given by (\ref{4Duality3}) is
onto. Note that according to (\ref{4Inter}), we have
\begin{equation}\label{4Interbis}
L^p(M)=[M,L^2(M)]_{\theta},
\end{equation}
where $\theta=\frac{2}{p}$. We will now check that for any integer
$N\geq 1$, we have
\begin{equation}\label{4duality7}
L^p\{M;R_N\}_\ell\,\simeq\,\bigl[M\minten R_N, L^2\{M;R_N
\}_\ell\bigr]_{\theta}\qquad\hbox{isometrically}.
\end{equation}
For that purpose, let $y\in V\otimes R_N$ and let
$\norm{y}_{\theta}$ denote its norm in the above  interpolation
space.

Assume that $\norm{y}_{\alpha_p^\ell}<1$. There exist $c,d\in V$
and $z\in M\otimes R_N$ such that $y=czd$, $\norm{z}_{\rm min}<1$,
$\norm{c}_\infty<1$ and $\norm{d}_p<1$. Consider the strip
$$
\Sigma=\{\lambda\in\Cdb\, :\, 0< {\rm Re}(\lambda)< 1\}.
$$
According to (\ref{4Interbis}), there exists a continuous function
$D\colon \overline{\Sigma}\to M+L^2(M)$ whose restriction to
$\Sigma$ is analytic, such that $D(\theta)=d$, the functions
$t\mapsto D(it)$ and $t\mapsto D(1+it)$ belong to $C_0(\Rdb;M)$
and $C_0(\Rdb;L^2(M))$ respectively, and such that
$\norm{D(it)}_\infty <1$ and $\norm{D(1+it)}_2<1$ for any
$t\in\Rdb$. We define
$$
f\colon \overline{\Sigma}\longrightarrow M\minten R_N +
L^2\{M;R_N\}_\ell
$$
by letting
$$
f(\lambda)\, =\, cz D(\lambda),\qquad \lambda\in
\overline{\Sigma}.
$$
Then $f$ is continuous, its restriction to $\Sigma$ is analytic
and we have $f(\theta)=y$. Moreover the functions $t\mapsto f(it)$
and $t\mapsto f(1+it)$ belong to $C_0(\Rdb;M\minten R_N)$ and
$C_0(\Rdb;L^2\{M;R_N\}_\ell)$, respectively. Further for any
$t\in\Rdb$ we have
$$
\norm{f(1+it)}_{\alpha_2^\ell}\leq
\norm{c}_\infty\minnorm{z}\norm{D(1+it)}_{2} <1
$$
by Lemma \ref{4Pisier3} (1). Also we have $\minnorm{f(it)}<1$ for
any $t\in\Rdb$, hence $\norm{y}_{\theta}<1$.

Assume conversely that $\norm{y}_{\theta}<1$ and write
$y=(y_1,\ldots,y_N)$. Thus there is an $N$-tuple $(f_1,\ldots,
f_N)$ of continuous functions from $\overline{\Sigma}$ into
$M+L^2(M)$ such that $f_n(\theta)=y_n$ and $f_{n\vert\Sigma}$ is
analytic for any $n=1,\ldots,N$, and such that
$$
\Bignorm{\sum_{n=1}^N f_n(it)\otimes e_n}_{M\minten R_N}
<1\qquad\hbox{and}\qquad \Bignorm{\sum_{n=1}^N f_n(1+it)\otimes
e_n}_{L^2\{M;R_N\}_\ell}<1
$$
for any $t\in\Rdb$. Let $a ,b\in V$ and  $z'_1,\ldots z'_N$ in $M$
such that
$$
\Bignorm{\sum_{n=1}^N z'_n\otimes e_n}_{M\minten R_N}<1,\quad
\norm{a}_q<1,\quad\hbox{and}\quad\norm{b}_2<1.
$$
Since $[L^2(M),M]_\theta=L^q(M)$, there is a continuous function
$A\colon\overline{\Sigma}\to M+L^2(M)$ whose restriction to
$\Sigma$ is analytic, such that $A(\theta)=a$ and for any
$t\in\Rdb$, $\norm{A(it)}_2 <1$ and $\norm{A(1+it)}_\infty <1$.
Consider $F\colon\overline{\Sigma}\to \Cdb$ defined by
$$
F(\lambda)\,=\,\sum_{n=1}^N\varphi\bigl( A(\lambda) z'_n b
f_n(\lambda)\bigr),\qquad \lambda\in\overline{\Sigma}.
$$
Then $F$ is a well-defined continuous function, whose restriction
to $\Sigma$ is analytic. For any $t\in\Rdb$, we have
$$
\vert F(1+it)\vert\,\leq \Bignorm{\sum_{n=1}^N f_n(1+it)\otimes
e_n}_{L^2\{M;R_N\}_\ell}\,\Bignorm{\sum_{n=1}^N A(1+it)z'_n
b\otimes e_n}_{L^2\{M;R_N\}_\ell},
$$
by the first part of the proof of this theorem. Thus $\vert
F(1+it)\vert<1$. Likewise, we have
$$
\vert F(it)\vert\,\leq \Bignorm{\sum_{n=1}^N f_n(it)\otimes
e_n}_{M\minten R_N}\,\Bignorm{\sum_{n=1}^N A(it)z'_n b\otimes
e_n}_{L^1\{M;R_N\}_\ell}<1
$$
for any $t\in\Rdb$. It therefore follows from the three lines
lemma that $\vert F(\theta)\vert<1$. Since
$$
F(\theta)\,=\,\sum_{n=1}^{N}\varphi\bigl(a z'_n b y_n\bigr)
$$
is the action of $y$ on $\sum_{n=1}^N az'_n b\otimes e_n$, this
shows that the norm of $y$ as an element of
$L^{p'}\{M;R_N\}_\ell^*$ is $\leq 1$. By (\ref{4Duality6}), this
means that $\norm{y}_{\alpha_p^\ell}\leq 1$.

\bigskip
We will conclude our proof of (2) by adapting some ideas from
\cite[Chapter 1]{P0}. We momentarily fix two integers $1<k<m$ and
we let $P\colon R_m\to R_m$ be the orthogonal projection onto
$R_k= {\rm Span}\{e_1,\ldots,e_k\}$. We let $\overline{P} =
I_{V}\otimes P$ on $V\otimes R_m$.  For any $y\in V\otimes R_m$,
we have
$$
\minnorm{y}\,\leq\,\bigl(\minnorm{\overline{P}(y)}^2\, +\,
\minnorm{(I-\overline{P})(y)}^2\bigr)^{\frac{1}{2}}.
$$
Indeed this assertion simply means that for any $y_1,\ldots, y_m$
in $M$, we have
$$
\Bignorm{\sum_{n=1}^m y_n y_n^*}^{\frac{1}{2}}\,\leq \Bigl(
\Bignorm{\sum_{n=1}^{k} y_n y_n^*} + \Bignorm{\sum_{n = k+1}^m y_n
y_n^*}\Bigr)^{\frac{1}{2}}.
$$
Moreover it is plain that
$$
\norm{y}_{\alpha_2^\ell}\,\leq\,\norm{\overline{P}(y)}_{\alpha_2^\ell}\,
+\, \norm{(I-\overline{P})y}_{\alpha_2^\ell}.
$$
Recall that $2<p<\infty$ and let $s>1$ be defined by $\frac{1}{s}
= \frac{1}{2} +\frac{1}{p}$. By interpolation, using
(\ref{4duality7}),  we deduce from above that the (well defined)
linear mapping
$$
(V\otimes R_k) \oplus (V\otimes [R_m\ominus R_{k}])\longrightarrow
V\otimes R_m
$$
taking any $(\overline{P}(y), y-\overline{P}(y))$ to $y$ extends
to a contraction
$$
 L^{p}\{M;R_k\}_\ell\, {\mathop{\oplus}\limits^s}\,
L^{p}\{M;R_{m}\ominus R_{k}\}_\ell\,\longrightarrow\,
L^{p}\{M;R_m\}_\ell.
$$
By (\ref{4Duality6}) its adjoint is a contraction
$$
L^{p'}\{M;R_m\}_\ell\,\longrightarrow\, L^{p'}\{M;R_k\}_\ell
\,{\mathop{\oplus}\limits^{s'}}\, L^{p'}\{M;R_{m}\ominus
R_{k}\}_\ell
$$
and this adjoint maps any $y'\in V\otimes R_m$ to the pair
$(\overline{P}(y'), y'-\overline{P}(y'))$.

We deduce that for any finite family $(y'_1,\ldots, y'_m)$ in
$L^{p'}(M)$ and any $1<k<m$, we have
\begin{equation}\label{4Convexity1}
\Bignorm{\sum_{n=1}^k y'_n\otimes e_n}^{s'}_{\alpha_{p'}^\ell}\,
+\, \Bignorm{\sum_{n=k+1}^m y'_n\otimes
e_n}^{s'}_{\alpha_{p'}^\ell}\,\leq \Bignorm{\sum_{n=1}^m
y'_n\otimes e_n}^{s'}_{\alpha_{p'}^\ell}.
\end{equation}
(It should be observed that $s'$ is finite.) Let $\zeta\in
L^{p}\{M;R\}_\ell^*$. For any integer $n\geq 1$, let
$\zeta_n\colon L^p(M)\to \Cdb$ be defined by $\zeta_n(y) =
\zeta(y\otimes e_n)$. Then $\zeta_n$ is represented by some
$y'_n\in L^{p'}(M)$, and it is easy to show, using the density of
$V\otimes \cup_m R_m$ in $L^{p}\{M;R\}_\ell$,  that
\begin{equation}\label{4Convexity2}
\norm{\zeta}_{L^{p}\{M;R\}_\ell^*}\,=\,
\lim_{m\to\infty}\Bignorm{\sum_{n=1}^m y'_n \otimes
e_n}_{\alpha_{p'}^\ell}.
\end{equation}
Letting $m\to\infty$ in (\ref{4Convexity1}), we deduce that
$$
\Bignorm{\sum_{n=1}^k y'_n\otimes e_n}^{s'}_{\alpha_{p'}^\ell}\,
+\, \Bignorm{\zeta - \sum_{n=1}^k y'_n\otimes
e_n}^{s'}_{L^p\{M,R\}_\ell^*}\,\leq\,\norm{\zeta}^{s'}_{L^p\{M,R\}_\ell^*}
$$
for any $k\geq 1$. Using (\ref{4Convexity2}) again, this implies
that
$$
\bignorm{\zeta - \sum_{n=1}^k y'_n\otimes
e_n}_{L^p\{M,R\}_\ell^*}\,\longrightarrow\, 0
$$
when $k\to\infty$. Thus $\zeta$ belongs to the closure of
$L^{p'}(M)\otimes R$, hence $\zeta\in L^{p'}\{M;R\}_\ell$.
\end{proof}

\begin{remark}\label{4Rem}

\

(1) The isometric embedding in Theorem \ref{4Duality} (1) is not
surjective in general. Indeed let $B=B(\ell^2)$ and set
$S^2\{R\}_\ell =L^{2}\{B;R\}_\ell$. As  in (\ref{2Pisier5}), we
have
\begin{equation}\label{4Rem1}
S^2\{R\}_\ell  \simeq C\otimes_h R\otimes_h C
\end{equation}
and passing to the opposite structures, this yields
$$
S^2\{R\}_\ell  \simeq R\otimes_h C\otimes_h R.
$$
Regard $S^{1}=B_*$ as the predual operator space of $B$. By
well-known computations, we deduce that $S^2\{R\}_\ell \simeq
S^{1}\otimes_h R$ and that $S^2\{R\}_\ell^{*} \simeq  B\otimes_h
C$. On the other hand, $S^2\{R\}_\ell\simeq S^\infty\otimes_h C$
by (\ref{4Rem1}). Hence the embedding of $S^2\{R\}_\ell$ into its
dual corresponds to  $\iota\otimes I_C$, where $\iota\colon
S^\infty\hookrightarrow B$ is the canonical embedding of the
compact operators into the bounded operators.

Likewise for any $1<p\leq 2$, the embedding of $S^{p'}\{R\}_\ell$
into $S^{p}\{R\}_\ell^*$ corresponds to
$$
\iota \otimes I_{R\bigl(\tfrac{2}{p'}\bigr)}\colon
S^\infty\otimes_h R\bigl(\tfrac{2}{p'}\bigr)\hookrightarrow
B\otimes_h R\bigl(\tfrac{2}{p'}\bigr).
$$

\smallskip (2) Let $F$ be an operator space, let $1<p<\infty$ and
suppose that
\begin{equation}\label{4Rem2}
L^{p'}\{M;F^{*\rm op}\}_\ell\,\longrightarrow\,
L^{p}\{M;F\}_{\ell}^{*}\quad\hbox{contractively}.
\end{equation}
Then we also have
\begin{equation}\label{4Rem3}
M\minten F^{*\rm op}\,\longrightarrow\,
L^{1}\{M;F\}_{\ell}^{*}\quad\hbox{contractively}.
\end{equation}
Indeed assume that $p\geq 2$, and let $2< q\leq\infty$ such that
$\frac{1}{p}+\frac{1}{q}=\frac{1}{2}$. Let $w\in M\otimes F^{*\rm
op}$ and let $y\in V\otimes F$ with $\norm{y}_{\alpha_1^\ell}<1$.
Then we can write $y=azb$ for some $a,b\in V$ and some $z\in
M\otimes F$ such that $\norm{a}_2<1$, $\norm{b}_2<1$ and
$\minnorm{z}<1$. Let us factorise $a$ and $b$ in the form
$a=a_1a_2$ and $b=b_1b_2$, with $a_1,a_2,b_1,b_2\in V$ verifying
$\norm{a_1}_2<1, \norm{a_2}_\infty<1,\norm{b_1}_p<1,
\norm{b_2}_q<1$. It is plain that
$$
\langle w,y\rangle=\langle w,azb\rangle=\langle
b_2wa_1,a_2zb_1\rangle.
$$
Hence by our assumption, we have
\begin{align*}
\vert \langle w,y\rangle\vert\,&\leq\,
\norm{b_2wa_1}_{L^{p'}\{M;F^{*\rm op}\}_\ell}\,
\norm{a_2zb_1}_{L^{p}\{M;F\}_\ell}\\ &\leq\,\norm{a_1}_2
\norm{a_2}_\infty\norm{b_1}_p\norm{b_2}_q\minnorm{w}\minnorm{z}\leq
\minnorm{w}.
\end{align*}
This shows (\ref{4Rem3}). It is a well-known consequence of
Haagerup's characterization of injectivity \cite{H2} that if the
von Neumann algebra $M$ is not injective, then (\ref{4Rem3}) does
not hold true for $F=\ell^{\infty}$. The above argument shows that
for any $1<p<\infty$, (\ref{4Rem2}) cannot hold true either in
this case.

\smallskip (3) Using a standard approximation argument, we deduce
from (\ref{4duality7}) that for any $p\geq 2$,
$$
L^p\{M;R\}_\ell\,\simeq\,\bigl[M\minten R, L^2\{M;R
\}_\ell\bigr]_{\frac{2}{p}}\qquad\hbox{isometrically}.
$$
Also, slightly modifying our arguments in the proof of Theorem
\ref{4Duality}, one can show that
$$
M\minten R_N \,\simeq\, L^1\{M;R_N\}_\ell^*
$$
for any $N\geq 1$. Details are left to the reader.

\smallskip (4) Lemma \ref{4Pisier3},
Theorem \ref{4Duality} and all formulas above have versions for
the `$r$-case', i.e. with the spaces $L^p\{M;F\}_r$ in place of
$L^p\{M;F\}_\ell$. These versions can be obtained by mimicking the
proofs of the `$\ell$-case', or by applying that `$\ell$-case'
together with (\ref{4Opposite}). Thus the `$r$-version' of Theorem
\ref{4Duality}  says that for any $1<p <\infty$, we have
$$
L^{p'}\{M;C\}_{r}\,\hookrightarrow\,
L^{p}\{M;C\}_r^{*}\quad\hbox{isometrically},
$$
and that this embedding is onto if $p>2$.

\end{remark}

\medskip
\section{Rigid factorizations and dilations of $L^p$ operators}
In this section we study various properties for bounded linear
maps on noncommutative $L^p$-spaces. We need to introduce the
matricial structure of $L^p(M)$. If $(M,\varphi)$ is any
semifinite von Neumann algebra, we equip $M_k(M)=M_k\otimes M$
with the trace $tr\otimes\varphi$ for any $k\geq 1$, where $tr$ is
the usual trace on $M_k$. This gives rise to the noncommutative
$L^p$-spaces $L^p(M_k(M))$. According to \cite[p. 141]{P2}, there
exists a (necessarily unique) operator space structure on $L^p(M)$
such that
$$
S^p_k[L^p(M)]\,\simeq\, L^p(M_k(M))\quad\hbox{isometrically}
$$
for any $k\geq 1$. (This structure is obtained by interpolation
between the predual operator space of $M^{\rm op}$ and $M$.)

We say that a linear map $u\colon L^p(M)\to L^p(M)$ is positive if
it maps $L^p(M)_+$ into itself. (Note that $L^p(M)$ is spanned by
$L^p(M)_+$.) Next we say that $u$ is completely positive if
$$
I_{S^p_k}\otimes u\colon L^p(M_k(M))\longrightarrow  L^p(M_k(M))
$$
is positive for any $k\geq 1$.

\bigskip
We will consider isometries on noncommutative $L^p$-spaces, and we
will use their description given by Yeadon's theorem (see also
Remark \ref{4Y2} below).

\begin{theorem}\label{4Y1}\, {\rm (Yeadon \cite{Y})}
Let $(M,\varphi)$ and $(N,\psi)$ be two semifinite von Neumann
algebras, let $1<p\not= 2<\infty$, and let $T\colon L^p(N)\to
L^p(M)$ be a linear isometry. There exist a one-to-one normal
Jordan homomorphism $J\colon N\to M$, a positive unbounded
operator $B$ affiliated with $J(N)'\cap M$ and a partial isometry
$W\in M$ such that $W^*W$ is the support projection of $B$,
$\psi(a) =\varphi(B^pJ(a))$ for all $a\in N_+$, and
$$
T(a)\,=\, WBJ(a),\qquad a\in N\cap L^p(N).
$$
\end{theorem}

\begin{remark}\label{4Y2}
We will need a little information on Jordan homomorphisms, for
which we refer e.g. to \cite[pp. 773-777]{KR}. Let $M,N$ be von
Neumann algebras. We recall that a Jordan homomorphism $J\colon
N\to M$ is a linear map satisfying $J(a^2)=J(a)^2$ and
$J(a^*)=J(a)^*$ for any $a\in N$. Assume that $J\colon N\to M$ is
a normal Jordan homomorphism, and let $D\subset M$ be the von
Neumann algebra generated by the range of $J$. Then there exists
two central projections $e_1,e_2$ of $D$ such that the map
$\pi_1\colon N\to M$ defined by $\pi_1(a)=J(a)e_1$ is a
$*$-representation, the map $\pi_2\colon N\to M$ defined by
$\pi_2(a)=J(a)e_2$ is a $*$-anti-representation, and $e_1+e_2$ is
equal to the unit of $D$. Thus we have $J=\pi_1+\pi_2$.
\end{remark}

\bigskip
Throughout the rest of this section, we fix a number $1<p\not =
2<\infty$, and we let $p'$ denote its conjugate number. Let
$(N,\psi)$ be a semifinite von Neumann algebra and let $u\colon
L^p(N)\to L^p(N)$ be a linear mapping. We say that $u$ admits a
rigid factorisation if there exist another semifinite von Neumann
algebra $(M,\varphi)$ and two linear isometries $T\colon L^p(N)\to
L^p(M)$ and $S\colon L^{p'}(N)\to L^{p'}(M)$ such that $u=S^*T$.
\begin{displaymath}
\xymatrix{ & L^p(M) \ar[rd]^{S^*} & \\  L^p(N) \ar[rr]^{u}
\ar[ru]^{T} & & L^p(N)}
\end{displaymath}

We note that any completely positive contraction $u\colon S^p_k\to
S^p_k$ is completely contractive. This follows from \cite[Prop.
2.2 and Lem. 2.3]{P1}. The main result of this section is the
following.

\begin{theorem}\label{4Main} Assume that $1<p\not =
2<\infty$. There exist an integer $k\geq 1$ and a completely
positive contraction $u\colon S^p_k\to S^p_k$ which does not have
a rigid factorisation.
\end{theorem}

The origin of this result is the search for a noncommutative
analog of Akcoglu's dilation theorem \cite{A,AS}. Let
$(\Omega,\mu)$ be a measure space, and let $u\colon L^p(\Omega)\to
L^p(\Omega)$ be a positive contraction. Akcoglu's theorem asserts
that there exist another measure space $(\Omega',\mu')$, two
contractions
$$
J\colon L^p(\Omega)\longrightarrow
L^p(\Omega')\qquad\hbox{and}\qquad Q\colon
L^p(\Omega')\longrightarrow L^p(\Omega),
$$
and an invertible isometry $U\colon L^p(\Omega')\to L^p(\Omega')$
such that $u^n=QU^nJ$ for any integer $n\geq 0$.
\begin{displaymath}
\xymatrix{ L^p(\Omega') \ar[r]^{U^n} & L^p(\Omega') \ar[d]^{Q}\\
L^p(\Omega) \ar[r]^{u^n} \ar[u]^{J} & L^p(\Omega)}
\end{displaymath}

Owing to that statement, we consider a noncommutative $L^p$-space
$L^p(N)$, a linear mapping $u\colon L^p(N)\to L^p(N)$, and we say
that $u$ is {\it dilatable} if there exist another noncommutative
$L^p$-space $L^p(M)$, two linear contractions $J\colon L^p(N)\to
L^p(M)$ and $Q\colon L^p(M)\to L^p(N)$, and an invertible isometry
$U\colon L^p(M)\to L^p(M)$ such that $u^n=QU^nJ$ for any integer
$n\geq 0$. Any dilatable operator is clearly a contraction and
Akcoglu's theorem implies that any positive contraction on a
commutative $L^p$-space is dilatable.

If $u\colon L^p(N)\to L^p(N)$ is a dilatable operator on a
noncommutative $L^p$-space, then $QJ$ is equal to the identity of
$L^p(N)$. Since $\norm{J}\leq 1$ and $\norm{Q}\leq 1$, this
implies that  $J$ and $Q^*$ are isometries. Furthermore we have
$u=QUJ$, hence $u=S^*T$, with $T=UJ$ and $S=Q^*$. This shows that
$u$ admits a rigid factorisation. As a consequence of Theorem
\ref{4Main}, we therefore obtain the following corollary, saying
that there is no direct analog of Akcoglu's theorem on
noncommutative $L^p$-spaces.

\begin{corollary}
For any $1<p\not = 2<\infty$, there is an integer $k\geq 1$ and a
completely positive contraction $u\colon S^p_k\to S^p_k$ which is
not dilatable.
\end{corollary}

We refer the reader to \cite{AD} for a related but different
notion of factorisation of linear maps as the product of an
isometry and of the adjoint of an isometry.

\bigskip
We will give two proofs of Theorem \ref{4Main}, one at the end of
this section and another one in Section 5. Both will rely on the
following decomposition result of independent interest.

\begin{proposition}\label{4Tensor}
Let $1<p\not= 2<\infty$ and let $(M,\varphi)$ and $(N,\psi)$ be
two semifinite von Neumann algebras. Let $T\colon L^p(N)\to
L^p(M)$ be a linear isometry. Then there exist two contractions
$T_1,T_2\colon L^p(N)\to L^p(M)$ such that
$$
T=T_1+T_2,
$$
and for any operator space $F$,
\begin{equation}\label{4Tensor1}
\bignorm{T_1\otimes I_{F}\colon L^p\{N;F\}_\ell\longrightarrow
L^p\{M;F\}_\ell}\leq 1
\end{equation}
and
\begin{equation}\label{4Tensor2}
\bignorm{T_2\otimes I_{F}\colon L^p\{N;F\}_\ell\longrightarrow
L^p\{M;F^{\rm op}\}_r}\leq 1.
\end{equation}
\end{proposition}

\begin{proof}
Let $T\colon L^p(N)\to L^p(M)$ be a linear isometry, and let
$W,B,J$ be provided by Yeadon's theorem \ref{4Y1}, so that
$T=WBJ$. We apply Remark \ref{4Y2} to the normal Jordan
homomorphism $J\colon N\to M$, and let $e_1,e_2,\pi_1,\pi_2$ be
given by this statement. Since $B$ commutes with the range of $J$,
it commutes with $e_1$, and hence $B$ commutes with the range of
$\pi_1$.

We define $T_1, T_2\colon  L^p(N)\to L^p(M)$ by letting
$$
T_1(a)\,=\, T(a)e_1\qquad\hbox{and}\qquad T_2(a)\,=\, T(a)e_2
$$
for any $a\in L^p(N)$. By construction, $T=T_1+T_2$.

\smallskip
Assume that $p< 2$ and let $q> 2$ be such that
$\frac{1}{2}+\frac{1}{q} \,=\,\frac{1}{p}$. Let $V=V(N)$ and let
$y\in V\otimes F$ such that $\norm{y}_{\alpha_p^\ell}<1$. Thus we
can write $y=azb$ for some $a,b\in V$ and $z\in N\otimes F$ such
that
$$
\norm{a}_q\leq 1,\quad \norm{b}_2\leq
1,\quad\hbox{and}\quad\minnorm{z}\leq 1.
$$
Let $(c_k)_k$ and $(x_k)_k$ be finite families in $N$ and $F$
respectively such that $z=\sum_k c_k\otimes x_k$. Then
$$
(T_1\otimes I_F)y\,=\, \sum_k T_1(ac_kb)\otimes x_k.
$$
Let $\theta=\frac{p}{2}$, so that $1-\theta=\frac{p}{q}$. Since
$\pi_1= J(\cdotp)e_1$ is a $*$-representation whose range commutes
with $B$, we have
$$
T_1(ac_kb) = WB\pi_1(ac_kb) = WB \pi_1(a)\pi_1(c_k)\pi_1(b) =
WB^{1-\theta}\pi_1(a)\pi_1(c_k)B^{\theta}\pi_1(b)
$$
for any $k$. Hence
\begin{align*}
(T_1\otimes I_F)y\, & =\, WB^{1-\theta}\pi_1(a)\,\Bigl(\sum_k
\pi_1(c_k)\otimes x_k\Bigr)\,B^{\theta}\pi_1(b)\\
& =\,WB^{1-\theta} \pi_1(a)\bigl(\pi_1\otimes
I_F\bigr)(z)B^{\theta}\pi_1(b).
\end{align*}
By Lemma \ref{4Pisier3}, we deduce that
$$
\bignorm{(T_1\otimes I_F)y}_{L^p\{M;F\}_\ell}\,\leq\,
\norm{WB^{1-\theta} \pi_1(a)}_q \minnorm{(\pi_1\otimes I_F)(z)}
\norm{B^{\theta}\pi_1(b)}_2.
$$
Since $W$ is the support projection of $B$, we have $\vert
WB^{1-\theta} \pi_1(a)\vert = \vert  B^{1-\theta} \pi_1(a)\vert$.
Since $B$ commutes with the range of $\pi_1$, and $\pi_1$ is a
$*$-representation, we deduce that
$$
\vert WB^{1-\theta} \pi_1(a)\vert^q =
B^{q(1-\theta)}\vert\pi_1(a)\vert^q = B^p\pi_1(\vert a\vert^q).
$$
Thus
$$
\norm{WB^{1-\theta} \pi_1(a)}_q^q = \varphi\bigl( B^p\pi_1(\vert
a\vert^q)\bigr)   \leq \varphi\bigl( B^p J(\vert a\vert^q)\bigr) =
\psi(\vert a\vert^q) = \norm{a}_q^q\leq 1.
$$
Likewise, we have
$$
\norm{B^{\theta}\pi_1(b)}_2\leq\norm{b}_2\leq 1.
$$
The $*$-representation $\pi_1$ is a complete contraction, hence
$$
\minnorm{(\pi_1\otimes I_F)(z)}\leq\minnorm{z}\leq 1.
$$
Thus we obtain that $\bignorm{(T_1\otimes
I_F)y}_{L^p\{M;F\}_\ell}\leq 1$. This shows (\ref{4Tensor1}), that
is, $T_1\otimes I_F$ extends to a contraction from
$L^p\{N;F\}_\ell$ into $L^p\{M;F\}_\ell$. The proof for $p\geq 2$
is similar.

The inequality (\ref{4Tensor2}) can be proved by similar
arguments. It also follows from the above proof and the
identification (\ref{4Opposite}). Indeed, saying that $\pi_2\colon
N\to M$ is an $*$-anti-representation means that $\pi_2$ is a
$*$-representation from $N$ into $M^{\rm op}$.
\end{proof}

\begin{remark}\label{4Complement}
Let $T, T_1,T_2\colon L^p(N)\to L^p(M)$ as above. Then we also
have
$$
\bignorm{T_1\otimes I_{F}\colon L^p\{N;F\}_r\longrightarrow
L^p\{M;F\}_r}\leq 1
$$
and
$$
\bignorm{T_2\otimes I_{F}\colon L^p\{N;F\}_r\longrightarrow
L^p\{M;F^{\rm op}\}_\ell}\leq 1
$$
for any operator space $F$. These estimates have the same proofs
as (\ref{4Tensor1}) and (\ref{4Tensor2}). Appealing to
(\ref{4Opposite}), they can be also viewed as a formal consequence
of the latter estimates.
\end{remark}

\bigskip
Our first proof of Theorem \ref{4Main} will appeal to
$L^p$-matricially normed spaces and some results from \cite{JLM}.
Let $X$ be a Banach space. For any integers $k,m\geq 1$ and any
$y\in S^p_k\otimes X$ and $y'\in S^p_m\otimes X$, let
$$
y\oplus y'\,=\,\left[\begin{array}{cc} y & 0\\ 0 &
y'\end{array}\right]
$$
denote the corresponding block diagonal element of
$S^{p}_{k+m}\otimes X$. Suppose that for any integer $k\geq 1$,
the matrix space $S^p_k\otimes X$ is equipped with a norm $\norm{\
}_{\alpha}$ and that the natural embedding $y\mapsto y\oplus 0$
from $S^p_k\otimes_\alpha X$ into $S^p_{k+1}\otimes_\alpha X$ is
an isometry. Here $S^p_k\otimes_\alpha X$ denotes the vector space
$S^p_k\otimes X$ equipped with the norm $\norm{\ }_{\alpha}$ and
by the above assumption, there is no ambiguity in the use of a
single notation $\norm{\ }_\alpha$ (not depending on $k$) for all
these matrix norms. We say that $X$ equipped with $\norm{\
}_\alpha$ is an $L^p$-matricially normed space if
$S^p_1\otimes_{\alpha}X=X$ isometrically and if the following two
properties hold.

\begin{itemize}
\item [(P1)] For any integer $k\geq 1$, for any $c,d\in M_{k}$ and
for any $y\in S^p_k\otimes X$, we have
$$
\norm{cyd}_\alpha\,\leq\,
\norm{c}_\infty\norm{y}_\alpha\norm{d}_\infty,
$$
where $\norm{\ }_\infty$ denotes the operator norm.

\item [(P2)] For any integers $k,m\geq 1$, and for any $y\in
S^p_k\otimes X$ and $y'\in S^p_m\otimes X$, we have
$$
\norm{y\oplus y'}_\alpha\, =\,\bigl(\norm{y}_\alpha^p
+\norm{y'}_\alpha^p\bigr)^{\frac{1}{p}}.
$$
\end{itemize}

Let $u\colon S^p_k\to S^p_k$ be a linear map. Following \cite{P1},
the regular norm of $u$, denoted by $\rnorm{u}$, is defined as the
smallest constant $K\geq 0$ such that
$$
\bignorm{u\otimes I_F\colon S^p_k[F]\longrightarrow S^p_k[F]} \leq
K
$$
for any operator space $F$.

\begin{theorem}\label{3Char1} (\cite{JLM})
Let $X,Y$ be two $L^p$-matricially normed spaces, with associated
norms on the matrix spaces $S^p_k\otimes X$ and $S^p_k\otimes Y$
denoted by $\norm{\ }_{\alpha}$ and $\norm{\ }_{\beta}$,
respectively. Let $\sigma\colon X\to Y$ be a bounded operator, and
assume that there is a constant $C\geq 0$ such that
\begin{equation}\label{3Char2}
\bignorm{u\otimes \sigma \colon S^p_k\otimes_\alpha
X\longrightarrow S^p_k\otimes_\beta Y}\,\leq\, C\rnorm{u}
\end{equation}
for any $u\colon S^p_k\to S^p_k$ and any $k\geq 1$. Then there
exist an operator space $F$ and two bounded operators
$$
\tau\colon X\longrightarrow F\qquad\hbox{and}\qquad \rho\colon
F\longrightarrow Y
$$
such that $\sigma=\rho\circ\tau$, $\tau$ has dense range and for
any $k\geq 1$,
\begin{equation}\label{3Char3}
\bignorm{I_{S^p_k}\otimes \tau \colon S^p_k\otimes_\alpha
X\longrightarrow S^p_k[F]}\leq C
\qquad\hbox{and}\qquad\bignorm{I_{S^p_{k}}\otimes \rho \colon
S^p_k[F] \longrightarrow S^p_k\otimes_\beta Y}\leq 1.
\end{equation}
\end{theorem}

\begin{remark}\label{2Sums}

\

(1)  Let $\norm{\ }_{\alpha_0}$  and  $\norm{\ }_{\alpha_1}$ be
norms on the matrix spaces $S^p_k\otimes X$ such that $X$ equipped
with $\norm{\ }_{\alpha_0}$ (resp. $\norm{\ }_{\alpha_1}$) is an
$L^p$-matricially normed space. We define a norm $\norm{\ }_\beta$
on each $S^p_k\otimes X$ by the following formula. For any $y\in
S^p_k\otimes X$,
$$
\norm{y}_\beta\,=\,\inf\bigl\{\bigl(\norm{y_0}_{\alpha_0}^p\,
+\,\norm{y_1}_{\alpha_1}^p\bigr)^{\frac{1}{p}}\, :\, y_0,y_1\in
S^p_k\otimes X, \ y=y_0+y_1\bigr\}.
$$
It turns out that $X$ equipped with $\norm{\ }_\beta$ is an
$L^p$-matricially normed space. This structure is obtained as the
`sum' of the ones given by $S^p_k\otimes_{\alpha_0} X$ and
$S^p_k\otimes_{\alpha_0} X$, and we simply write
$$
S^p_k\otimes_\beta X \, =\, S^p_k\otimes_{\alpha_0} X\, +_p\,
S^p_k\otimes_{\alpha_1} X
$$
in this case.

It is obvious that $\norm{\ }_\beta$ satisfies (P1) and the
inequality ``$\leq$" in (P2). To prove the reverse inequality
``$\geq$" in (P2), take $y\in S_k^p\otimes X$ and $y'\in
S_m^p\otimes X$  and assume that
$$
\left\Vert\left[\begin{array}{cc} y & 0\\ 0 & y'\end{array}\right]
\right\Vert_\beta <1.
$$
Then there exists a decomposition
$$
\left[\begin{array}{cc} y & 0\\ 0 & y'\end{array}\right] \, =\,
\left[\begin{array}{cc} y^{0}_{11} & y^{0}_{12} \\ y^{0}_{21} &
y^{0}_{22} \end{array}\right] \, +\,
\left[\begin{array}{cc} y^{1}_{11}  & y^{1}_{12} \\
y^{1}_{21}  & y^{1}_{22} \end{array} \right]\quad\hbox{with}\quad
\left\Vert\left[\begin{array}{cc} y^{0}_{11} & y^{0}_{12} \\
y^{0}_{21} & y^{0}_{22}
\end{array}\right]\right\Vert_{\alpha_0}^p\,+\, \left\Vert
\left[\begin{array}{cc} y^{1}_{11} & y^{1}_{12} \\
y^{1}_{21} & y^{1}_{22}
\end{array}\right]\right\Vert_{\alpha_1}^p\, <1.
$$
Since
$$
\left[\begin{array}{cc} y^{0}_{11} & 0 \\
0 & y^{0}_{22}
\end{array}\right]\,=\,\frac{1}{2}\Biggl(
\left[\begin{array}{cc} y^{0}_{11} & y^{0}_{12} \\
y^{0}_{21} & y^{0}_{22}
\end{array}\right]\, +\, \left[\begin{array}{cc} I_k & 0 \\
0& -I_m
\end{array}\right]\left[\begin{array}{cc} y^{0}_{11} & y^{0}_{12} \\
y^{0}_{21} & y^{0}_{22}
\end{array}\right]\left[\begin{array}{cc}I_k & 0 \\
0& -I_m
\end{array}\right]\Biggr),
$$
we obtain by applying (P1) and (P2) to $\norm{\ }_{\alpha_0}$ that
$$
\norm{y^{0}_{11}}_{\alpha_0}^p +
\norm{y^{0}_{22}}_{\alpha_0}^p\,\leq \,
\left\Vert\left[\begin{array}{cc} y^{0}_{11} & y^{0}_{12} \\
y^{0}_{21} & y^{0}_{22}
\end{array}\right]\right\Vert_{\alpha_0}^p.
$$
Similarly,
$$
\norm{y^{1}_{11}}_{\alpha_1}^p +
\norm{y^{1}_{22}}_{\alpha_1}^p\,\leq \,
\left\Vert\left[\begin{array}{cc} y^{1}_{11} & y^{1}_{12} \\
y^{1}_{21} & y^{1}_{22}
\end{array}\right]\right\Vert_{\alpha_1}^p.
$$
Since $y=y^{0}_{11} + y^{1}_{11}$ and $y'=y^{0}_{22} +
y^{1}_{22}$, we deduce that
$$
\norm{y}_{\beta}^p + \norm{y'}_{\beta}^p\,\leq\,
\norm{y^{0}_{11}}_{\alpha_0}^p + \norm{y^{0}_{22}}_{\alpha_0}^p +
\norm{y^{1}_{11}}_{\alpha_1}^p +
\norm{y^{1}_{22}}_{\alpha_1}^p\,<1,
$$
which proves the desired inequality.

\smallskip
(2) Let $F$ be an operator space and recall that we have
$$
S^p_k\{F\}_\ell = S^p_k\otimes_{\alpha_p^\ell}
F\qquad\hbox{and}\qquad S^p_k\{F\}_r = S^p_k\otimes_{\alpha_p^r}
F.
$$
According to \cite[Section 2]{JLM}, $F$ equipped with $\norm{\
}_{\alpha_p^\ell}$ (resp. $\norm{\ }_{\alpha_p^r}$) is an
$L^p$-matricially normed space. In the sequel we will use the
$L^p$-matricially normed space structure on $\ell^2$ defined as
the sum of $S_k^p\{R\}_\ell$ and $S_k^p\{C\}_r$.
\end{remark}

The following is independent of Theorem \ref{3Char1} and will be
used in both proofs of Theorem \ref{4Main}

\begin{corollary}\label{4Rigid}
Let $1<p\not= 2<\infty$ and suppose that  $u\colon S^p_k\to S^p_k$
admits a rigid factorisation. Then
$$
\bignorm{u\otimes I_{\ell^2}\colon S^p_k\{R\}_\ell \longrightarrow
S_k^p\{R\}_\ell \, +_p\, S_k^p\{C\}_r}\leq 4.
$$
\end{corollary}

\begin{proof}
Suppose that $u\colon S^p_k\to S^p_k$ admits a rigid
factorisation. By definition there exist a semifinite von Neumann
algebra $M$ and two linear isometries
$$
T\colon S^p_k\longrightarrow L^p(M)\qquad\hbox{and}\qquad S\colon
S^{p'}_k\longrightarrow L^{p'}(M)
$$
such that $u=S^*T$. According to Proposition \ref{4Tensor}, we
have a decomposition $T=T_1+T_2$ for some $T_1,T_2\colon S^p_k\to
L^p(M)$ satisfying
$$
\bignorm{T_1\otimes I_{F}\colon S^p_k\{F\}_\ell\to
L^p\{M;F\}_\ell}\leq 1  \quad\hbox{and}\quad \bignorm{T_2\otimes
I_{F}\colon S^p_k\{F\}_\ell\to L^p\{M;F^{\rm op}\}_r}\leq 1
$$
for any operator space $F$. Likewise we have a decomposition
$S=S_1+S_2$ for some $S_1,S_2\colon S^{p'}_k\to L^{p'}(M)$
satisfying
$$
\bignorm{S_1\otimes I_{G}\colon S^{p'}_k\{G\}_\ell \to
L^{p'}\{M;G\}_\ell}\leq 1  \quad\hbox{and}\quad
\bignorm{S_2\otimes I_{G}\colon S^{p'}_k\{G\}_\ell \to
L^p\{M;G^{\rm op}\}_r}\leq 1
$$
for any operator space $G$. By Remark \ref{4Complement}, we also
have
$$
\bignorm{S_1\otimes I_{G}\colon S^{p'}_k\{G\}_r\to
L^{p'}\{M;G\}_r}\leq 1  \quad\hbox{and}\quad \bignorm{S_2\otimes
I_{G}\colon S^{p'}_k\{G\}_r\to L^p\{M;G^{\rm op}\}_\ell}\leq 1.
$$
Mixing the two decompositions, we have
$$
u\,=\, S_1^{*}T_1\,+\,S_2^{*}T_1\,+\,S_1^{*}T_2\,+\,S_2^{*}T_2.
$$
Since $S_1\otimes I_{R}$ is a contraction from
$S^{p'}_k\{R\}_\ell$ into $L^{p'}\{M;R\}_\ell$, it follows from
Theorem \ref{4Duality} that $S_1^{*}\otimes I_{R}$ extends to a
contraction from $L^{p}\{M;R\}_\ell$ into $S^{p}_k\{R\}_\ell$.
Consequently,
$$
\bignorm{S_1^{*}T_1\otimes I_{R}\colon S^p_k\{R\}_\ell
\longrightarrow S^p_k\{R\}_\ell}\leq 1.
$$
Likewise, since $S_2\otimes I_{R}$ is a contraction from
$S^{p'}_k\{C\}_r$ into $L^{p'}\{M;R\}_\ell$, it follows from
Theorem \ref{4Duality} and Proposition \ref{2Duality2} that
$S_2^{*}\otimes I_{R}$ extends to a contraction from
$L^{p}\{M;R\}_\ell$ into $S^{p}_k\{C\}_r$. Consequently,
$$
\bignorm{S_2^{*}T_1\otimes I_{R}\colon
S^p_k\{R\}_\ell\longrightarrow S^p_k\{C\}_r}\leq 1.
$$
Similarly we obtain that
$$
\bignorm{S_1^{*}T_2\otimes I_{R}\colon S^{p}_k\{R\}_\ell\to
S^{p}_k\{C\}_r}\leq 1  \quad\hbox{and}\quad
\bignorm{S_2^{*}T_2\otimes I_{R}\colon S^{p}_k\{R\}_\ell\to
S^{p}_k\{R\}_\ell}\leq 1.
$$
The result follows at once.
\end{proof}

\begin{proof} (Of Theorem \ref{4Main}.)
By duality we may suppose that $p>2$. Following Remark
\ref{2Sums}, let $\norm{\ }_\beta$ denote the matrix norms on
$\ell^2$ given by
$$
S^p_k\otimes_\beta \ell^2 \, =\, S_k^p\{R\}_\ell \, +_p\,
S_k^p\{C\}_r.
$$
Assume that for any integer $k\geq 1$, every completely positive
contraction $S^p_k\to S^p_k$ admits a rigid factorisation. Let
$u\colon S^p_k\to S^p_k$ be an arbitrary linear map. By \cite{P1}
and \cite[Cor. 8.7]{P5}, one can find four completely positive
maps $u_1,u_2,u_3,u_4\colon S^p_k\to S^p_k$ such that $u=(u_1-u_2)
+i(u_3-u_4)$ and for any $j=1,\ldots,4$,
$\norm{u_j}\leq\rnorm{u}$. By Corollary \ref{4Rigid} we deduce
that
$$
\bignorm{u\otimes I_{\ell^2}\colon S^p_k\{R\}_\ell\longrightarrow
S^p_k\otimes_\beta \ell^{2}}\leq 16\rnorm{u}.
$$
Let us apply Theorem \ref{3Char1} with $X=Y=\ell^{2}$, and
$\sigma=I_{\ell^{2}}$. Thus there exist an operator space $F$ and
two bounded operators $\tau\colon\ell^{2}\to F$ and $\rho\colon
F\to \ell^{2}$ such that $\rho\circ\tau=I_{\ell^{2}}$ and for any
$k\geq 1$,
$$
\bignorm{I_{S^p_k}\otimes \tau \colon
S^p_k\{R\}_\ell\longrightarrow S^p_k[F]}\leq 16
\qquad\hbox{and}\qquad\bignorm{I_{S^p_{k}}\otimes \rho \colon
S^p_k[F] \longrightarrow S^p_k\otimes_\beta \ell^{2}}\leq 1.
$$
Moreover we can assume that $F$ is equal to the range of $\tau$
and hence,  $\rho=\tau^{-1}$. We can now conclude and get to a
contradiction as in the proof of \cite[Theorem 2.6]{JLM}. We only
give a sketch of the argument and refer the reader to the latter
paper for details.

By means of (\ref{2Pisier5}) and (\ref{2Pisier8}), the above
estimates imply that
$$
\norm{\tau^{-1}}\leq 1\qquad\hbox{and}\qquad
\bignorm{I_{\ell^2_k}\otimes \tau \colon C_k \otimes_h R
\longrightarrow R_k\bigl(1-\tfrac{1}{p}\bigr)\otimes_h F}\,\leq\,
16
$$
for any $k\geq 1$. Using the well-known isometric identifications
$$
C_k\otimes_h R_k \simeq M_k\qquad\hbox{and}\qquad
CB\bigl(C_k,R_k\bigl(1-\tfrac{1}{p}\bigr)\bigr)\simeq S^{2p}_k,
$$
we can deduce that $\norm{v}_{2p}\leq 16\norm{v}_\infty$ for any
linear mapping $v\colon\ell^2_k\to\ell^2_k$. This is false if
$k>16^{2p}$.
\end{proof}

\begin{remark}
So far we have only considered noncommutative $L^p$-spaces
associated with a semifinite trace. In fact semifiniteness was
necessary to define the spaces $L^p\{M;F\}_\ell$ (or
$L^p\{M;F\}_r$), and hence the duality results stated in Section 3
make sense only in the tracial setting. We wish to indicate
however that Corollary \ref{4Rigid} and Theorem \ref{4Main} extend
to the non tracial case.

More precisely, let $M$ be an arbitrary von Neumann algebra and
for any $1\leq p\leq \infty$, let $L^p(M)$ denote the
noncommutative $L^p$-space constructed by Haagerup \cite{H0}. We
refer the reader to \cite{T} for a complete description of these
spaces, and to \cite{PX2} or \cite{JX} for a brief presentation.
We recall that if $M$ is semifinite and $\varphi$ is a n.s.f.
trace on $M$, then Haagerup's space $L^p(M)$ is isometrically
isomorphic to the usual tracial $L^p$-space (see Section 2). Our
extension of Corollary \ref{4Rigid} is as follows: for any
$1<p\not=2<\infty$, for any integer $k\geq 1$ and for any pair of
isometries
\begin{equation}\label{4NonT}
T\colon S^p_k\longrightarrow L^p(M)\qquad\hbox{and}\qquad S\colon
S^{p'}_k\longrightarrow L^{p'}(M),
\end{equation}
we have
$$
\bignorm{S^*T\otimes I_{\ell^2}\colon S^p_k\{R\}_\ell
\longrightarrow S_k^p\{R\}_\ell \, +_p\, S_k^p\{C\}_r}\leq 4.
$$
Likewise, Theorem \ref{4Main} extends as follows: for $k\geq 1$
large enough, there exists a completely positive contraction
$u\colon S^p_k\to S^p_k$ such that whenever $M$ is a (not
necessarily semifinite) von Neumann algebra there is no pair
$(T,S)$ of isometries as in (\ref{4NonT}) such that $u=S^*T$.

The proofs of these extensions are similar to the ones given above
in the tracial case, up to technical details. They require the
extension of Yeadon's theorem obtained in \cite[Th. 3.1]{JRS} as
well as the duality techniques from \cite[Section 1]{JX}. We skip
the details.
\end{remark}

\begin{remark}
Let $(\Omega,\mu)$ be a measure space and let $u\colon
L^p(\Omega)\to L^p(\Omega)$ be a contraction (with
$1<p\not=2<\infty$). The following assertions are equivalent.

\begin{itemize}
\item [(i)] $u$ admits a rigid factorisation.

\item [(ii)] There exist a measure space $(\Omega',\mu')$ and two
linear isometries $T\colon L^{p}(\Omega)\to L^{p}(\Omega')$ and
$S\colon L^{p'}(\Omega)\to L^{p'}(\Omega')$ such that $u=S^*T$
({\it commutative} rigid factorisation).

\item [(iii)] For any integer $k\geq 1$,
$$
\bignorm{u\otimes I_{\ell^\infty_k}\colon
L^p(\Omega;\ell^\infty_k) \longrightarrow
L^p(\Omega;\ell^\infty_k)} \leq 1.
$$
(Equivalently, $u$ is regular and $\rnorm{u}\leq 1$, see
\cite{P1}.)

\item [(iv)] There exists a positive contraction $v$ on
$L^p(\Omega)$ such that $\vert u(f)\vert\leq v(\vert f\vert)$ for
any $f\in L^p(\Omega)$.
\end{itemize}

The equivalence of (ii) and (iv) follows from \cite[Section
3]{Peller}, and the equivalence of (iii) and (iv) is well-known
(see e.g. \cite{MN}). So we only need to show that (i) implies
(iii). For this purpose, assume that $u=S^*T$, where $T\colon
L^p(\Omega)\to L^p(M)$ and $S\colon L^{p'}(\Omega)\to L^{p'}(M)$
are isometries. For any integer $k\geq 1$, let
$L^p(M;\ell^\infty_k)$ and $L^{p'}(M;\ell^1_k)$ be the operator
space valued spaces introduced in \cite{J}. Arguing as in the
proof of Proposition \ref{4Tensor} it is not hard to show that
$$
T\otimes I_{\ell^\infty_k}\colon L^p(\Omega;\ell^\infty_k)
\longrightarrow L^p(M;\ell^\infty_k)\qquad\hbox{and}\qquad
S\otimes I_{\ell^1_k}\colon L^{p'}(\Omega;\ell^1_k)
\longrightarrow L^{p'}(M;\ell^1_k)
$$
are contractions. Using \cite[Prop. 3.6]{J}, we deduce that
$u\otimes I_{\ell^\infty_k}$ is a contraction on
$L^p(\Omega;\ell^\infty_k)$.
\end{remark}

\medskip
\section{A concrete example}
The proof of Theorem \ref{4Main} given above has a serious
drawback. Indeed, it does not show any concrete example of a
completely positive contraction $u\colon S^p_k\to S^p_k$ without a
rigid factorisation. The aim of this section is to present such an
example, thus giving another proof of that theorem. This second
proof does not use Theorem \ref{3Char1}.

\bigskip
Throughout we let $1<p<\infty$, we consider an integer $k\geq 1$.
Let $u_1\colon S^p_k\to S^p_k$ be defined by letting $u_1(E_{i1})=
k^{-\frac{1}{2p}}E_{ii}$ for any $i\geq 1$  and $u_1(E_{ij})= 0$
for any $j\geq 2$ and any $i\geq 1$. This can be written as
$$
u_1(x)\,=\,\sum_{i=1}^k a_{i}^{*} x b_i,\qquad x\in S^p_k,
$$
where
$$
a_i=E_{ii}\qquad\hbox{and}\qquad b_i= k^{-\frac{1}{2p}}
E_{1i},\qquad 1\leq i\leq k.
$$
Consider the three linear maps $u_2,u_3,u_4\colon S^p_k\to S^p_k$
defined by letting
$$
u_2(x)\,=\,\sum_{i=1}^k b_{i}^{*} x a_i,\quad
u_3(x)\,=\,\sum_{i=1}^k a_{i}^{*} x a_i,\quad\hbox{and}\quad
u_4(x)\,=\,\sum_{i=1}^k b_{i}^{*} x b_i
$$
for any $x\in S^p_k$. Then $u_3$ is the canonical diagonal
projection taking any $x=[x_{ij}]\in S^p_k$ to the diagonal matrix
$\sum_i x_{ii} E_{ii}$. Thus $\norm{u_3}=1$. Next, $u_4$ is the
rank one operator taking any $x=[x_{ij}]\in S^p_k$ to
$k^{-\frac{1}{p}} x_{11} I_k$, where $I_k$ denotes the identity
matrix. Since $\norm{I_k}_p= k^{\frac{1}{p}}$, we have
$\norm{u_4}=1$. According to \cite[Theorem 8.5]{P5} and \cite{P1},
this implies that $\rnorm{u_1}\leq 1$. In particular, $u_1$ is a
contraction. Likewise, $u_2$ is a contraction.

We now consider the average
\begin{equation}\label{5defu1}
u=\frac{1}{4}\,\bigl(u_1+u_2+u_3+u_4\bigr)
\end{equation}
of these four maps. Then $u\colon S^p_k\to S^p_k$ is a
contraction. Moreover we have
\begin{equation}\label{5defu2}
u(x)\,=\,\frac{1}{4}\,\sum_{i=1}^k (a_{i} +b_i)^{*} x (a_i
+b_i),\qquad x\in S^p_k.
\end{equation}
Hence $u$ is completely positive.

\begin{theorem}\label{5Main}
Assume that $1<p<\infty$, and let $u\colon S^p_k\to S^p_k$ be the
completely positive contraction defined by (\ref{5defu1}) and/or
(\ref{5defu2}).
\begin{enumerate}
\item [(1)] We have
$$
\lim_{k\to\infty}\,\bignorm{u\otimes I_{\ell^2_k}\colon
S_k^p\{R_k\}_\ell\longrightarrow S_k^p\{R_k\}_\ell \,+_p\,
S_k^p\{C_k\}_r}\,=\infty.
$$
\item [(2)] Assume that $p\not=2$. Then for $k$ large enough, the
operator $u$ does not admit a rigid factorisation.
\end{enumerate}
\end{theorem}

The proof will be given at the end of this section. We need the
following elementary lemma.

\begin{lemma}\label{5Elem}
Let $E_1$ and $E_2$ be two operator spaces with a common finite
dimension $k$. Let $(e_1^1,\ldots, e_k^1)$ and  $(e_1^2,\ldots,
e_k^2)$ be some bases of $E_1$ and $E_2$ respectively. Assume that
these bases are {\it completely $1$-unconditional}, in the sense
that for any $k$-tuple $\varepsilon=
(\varepsilon_1,\ldots,\varepsilon_k)$ with $\varepsilon_i=\pm 1$,
the operators
$$
V^1_\varepsilon\colon E_1 \longrightarrow
E_1\qquad\hbox{and}\qquad V^2_\varepsilon \colon E_2
\longrightarrow  E_2
$$
defined by letting $V^1_\varepsilon(e^1_i) = \varepsilon_i e^1_i$
and $V^2_\varepsilon(e^2_i) = \varepsilon_i e^2_i$ for any $1\leq
i\leq k$ are completely contractive. Let
$$
\Delta\colon E_1\otimes_h E_2\longrightarrow  E_1\otimes_h E_2
$$
be the `diagonal' projection defined by letting
$\Delta(e^1_i\otimes e^2_j)= 0$ if $i\not=j$, and
$\Delta(e^1_i\otimes e^2_i)=e^1_i\otimes e^2_i$ for any $i\geq 1$.
Then $\Delta$ is a complete contraction.
\end{lemma}

\begin{proof}
Let $\mu$ be the uniform probability measure on
$\Omega=\{-1,1\}^k$. It is easy to check that
$$
\Delta\,=\,\int_\Omega V^1_\varepsilon\otimes V^2_\varepsilon\,
d\mu(\varepsilon)\,.
$$
For any $\varepsilon\in\Omega$, we have
$$
\bigcbnorm{V^1_\varepsilon\otimes V^2_\varepsilon\colon
E_1\otimes_h E_2\longrightarrow  E_1\otimes_h E_2}\leq \cbnorm{
V^1_\varepsilon}\cbnorm{V^2_\varepsilon}\leq 1.
$$
Hence
$$
\cbnorm{\Delta}\leq \int_\Omega \cbnorm{V^1_\varepsilon\otimes
V^2_\varepsilon}\, d\mu(\varepsilon)\,\leq 1.
$$
\end{proof}

We let
$$
D_k\subset \ell^2_k\otimes  \ell^2_k \otimes \ell^2_k
$$
be the $k$-dimensional subspace of $\ell^2_k\otimes  \ell^2_k
\otimes \ell^2_k$ spanned by $\{e_i\otimes e_i\otimes e_i\, :\,
1\leq i\leq k\}$. Then we let
$$
P\colon \ell^2_k\otimes  \ell^2_k \otimes \ell^2_k\longrightarrow
\ell^2_k\otimes  \ell^2_k \otimes \ell^2_k
$$
be the projection onto $D_k$ defined by letting $P(e_i\otimes
e_j\otimes e_m)= 0$ if ${\rm card}\{i,j,m\}\geq 2$, and
$P(e_i\otimes e_i\otimes e_i)= e_i\otimes e_i\otimes e_i$ for any
$i\geq 1$. If $p\geq 2$, then according to the identification
\begin{equation}\label{5Haag}
S_k^p\{R_k\}_\ell\,=\, C_k\otimes_h R_k\otimes_h
R_k\bigl(\tfrac{2}{p}\bigr)
\end{equation}
given by (\ref{2Pisier5}), we may regard $P$ as defined on
$S_k^p\{R_k\}_\ell$. Using (\ref{2Pisier6}), we can do the same
when $p<2$.

\begin{lemma}\label{5Diag}
We have
$$
\bignorm{P\colon S_k^p\{R_k\}_\ell\longrightarrow
S_k^p\{R_k\}_\ell\,}=1.
$$
Moreover, for any complex numbers $\lambda_1,\ldots,\lambda_k$, we
have
$$
\Bignorm{\sum_{i=1}^k \lambda_i\, e_i\otimes e_i\otimes
e_i}_{S_k^p\{R_k\}_\ell}\,=\,
\Bigl(\sum_{i=1}^k\vert\lambda_i\vert^p\Bigr)^{\frac{1}{p}}.
$$
\end{lemma}

\begin{proof}
We assume that $p\geq 2$, the proof for $p<2$ being similar. Let
$$
\Delta\colon \ell^2_k\otimes \ell^2_k \longrightarrow
\ell^2_k\otimes \ell^2_k
$$
be the diagonal projection (in the sense of Lemma \ref{5Elem}).
Then we can write
\begin{equation}\label{5Comp}
P\, =\, (\Delta\otimes I_{\ell^2_k})\circ(I_{\ell^2_k}\otimes
\Delta),
\end{equation}
which is going to lead us to a two step proof.

We need several elementary operator space results, for which we
refer e.g. to \cite[Chapter 5]{P2} or \cite[Section 9.3]{ER2}.
First, $C_k\otimes_h R_k\simeq M_k$, and the diagonal of
$C_k\otimes_h R_k$ coincides with the commutative $C^*$-algebra
$\ell^\infty_k$. Second, $R_k\otimes_h C_k\simeq M_k^* = S^1_k$,
and the diagonal of $R_k\otimes_h C_k$ coincides with the operator
space dual of $\ell^\infty_k$, that is ${\rm Max}(\ell^1_k)$ (see
e.g.  \cite[Chapter 3]{P2}). Third, $R_k\otimes_h R_k\simeq
R_{k^2}$. We deduce form above that
$$
\bigcbnorm{\Delta\colon R_k\otimes_h R_k\to  R_k\otimes_h
R_k}=1\qquad\hbox{and}\qquad \bigcbnorm{\Delta\colon R_k\otimes_h
C_k\to R_k\otimes_h C_k}=1
$$
and moreover,
$$
\Delta(R_k\otimes_h R_k)\simeq R_k \qquad\hbox{and}\qquad
\Delta(R_k\otimes_h C_k)\simeq {\rm Max}(\ell^1_k)
$$
completely isometrically.

Next according to \cite[Theorem 5.22]{P2}, we have
$$
R_k\otimes_h R_k\bigl(\tfrac{2}{p}\bigr)\,\simeq\, \bigl[
R_k\otimes_h R_k,R_k\otimes_h C_k\bigr]_{\tfrac{2}{p}}
$$
completely isometrically. Hence by interpolation,
\begin{equation}\label{5Q1}
\bigcbnorm{\Delta\colon R_k\otimes_h
R_k\bigl(\tfrac{2}{p}\bigr)\longrightarrow R_k\otimes_h
R_k\bigl(\tfrac{2}{p}\bigr)} = 1
\end{equation}
and we have
\begin{equation}\label{5Q2}
\Delta\bigl(R_k\otimes_h
R_k\bigl(\tfrac{2}{p}\bigr)\bigr)\,\simeq\, \bigl[R_k, {\rm
Max}(\ell^1_k)\bigr]_{\tfrac{2}{p}}
\end{equation}
completely isometrically.

Now applying Lemma \ref{5Elem} with $E_1=C_k$ and $E_2={\rm
Max}(\ell^1_k)$, we find that
$$
\bigcbnorm{\Delta\colon C_k\otimes_h {\rm
Max}(\ell^1_k)\longrightarrow C_k\otimes_h {\rm Max}(\ell^1_k)}=1.
$$
We claim that
$$
\Delta\bigl(C_k\otimes_h {\rm Max}(\ell^1_k)\bigr)\,\simeq\,
\ell^2_k
$$
isometrically. Indeed, we have $C_k\otimes_h {\rm Max}(\ell^1_k) =
C_k\minten {\rm Max}(\ell^1_k)\simeq CB(\ell^\infty_k,C_k)$. Hence
writing $B=B(\ell^{2})$ for simplicity, we have for any
$\lambda_1,\ldots,\lambda_k$ in $\Cdb$ that
\begin{align*}
\Bignorm{\sum_{i=1}^k\lambda_i\, e_i\otimes e_i}_{C_k\otimes_h
{\rm Max}(\ell^1_k)}\,
&=\,\sup\Bigl\{\Bignorm{\sum_{i=1}^k\lambda_i  e_i\otimes
y_i}_{C_k\minten B}\, :\, y_i\in B,\
\sup_i\norm{y_i}\leq 1\,\Bigr\}\\
&=\,\sup\Bigl\{\Bignorm{\sum_{i=1}^k \vert\lambda_i\vert^2 y_i^*
y_i}_{B}^{\frac{1}{2}}\, :\, y_i\in B,\
\sup_i\norm{y_i}\leq 1\,\Bigr\}\\
&=\,\Bigl(\sum_{i=1}^k \vert\lambda_i\vert^2\Bigr)^{\frac{1}{2}}.
\end{align*}
On the other hand,
$$
\bigcbnorm{\Delta\colon C_k\otimes_h R_k \longrightarrow
C_k\otimes_h R_k}=1\qquad\hbox{and}\qquad \Delta\bigl(C_k\otimes_h
R_k\bigr)\,\simeq\, \ell^\infty_k.
$$
Since
$$
C_k\otimes_h [R_k,{\rm Max}(\ell^1_k)]_{\frac{2}{p}} =
\bigl[C_k\otimes_h R_k, C_k\otimes_h {\rm
Max}(\ell^1_k)\bigr]_{\frac{2}{p}},
$$
we deduce by interpolation that
\begin{equation}\label{5Q3}
\bigcbnorm{\Delta\colon C_k\otimes_h [R_k,{\rm
Max}(\ell^1_k)]_{\frac{2}{p}}\longrightarrow C_k\otimes_h
[R_k,{\rm Max}(\ell^1_k)]_{\frac{2}{p}}\,}=1.
\end{equation}
Since $[\ell^\infty_k,\ell^2_k]_{\frac{p}{2}}=\ell^p_k$, we obtain
in addition that
\begin{equation}\label{5Q4}
\Delta\bigl(C_k\otimes_h [R_k,{\rm
Max}(\ell^1_k)]_{\frac{2}{p}}\bigr)\,\simeq\, \ell^p_k
\end{equation}
isometrically.

Using (\ref{5Haag}) and the composition formula (\ref{5Comp}), we
deduce from (\ref{5Q1}), (\ref{5Q2}), (\ref{5Q3}) and (\ref{5Q4})
that $P$ is a contraction on $S^p_k\{R_k\}_\ell$, and that its
range is equal to $\ell^p_k$.
\end{proof}

\begin{proof} (Of Theorem \ref{5Main}.) The assertion (2) follows
from (1) by Corollary \ref{4Rigid}, so we only need to prove (1).
As in Section 4, we let
$$
S_k^p\otimes_\beta \ell^2_k\,=\, S_k^p\{R_k\}_\ell \,+_p\,
S_k^p\{C_k\}_r.
$$
We observe that Lemma \ref{5Diag} holds as well with
$S^p_k\{C_k\}_r$ replacing $S^p_k\{R_k\}_\ell$. Namely, $P$ is
contractive on $S^p_k\{C_k\}_r$, and $P(S^p_k\{C_k\}_r)$ is equal
to $\ell^p_k$. We deduce that
\begin{equation}\label{5NormP}
\bignorm{P\colon S_k^p\otimes_\beta \ell^2_k \longrightarrow
S_k^p\otimes_\beta \ell^2_k }=1
\end{equation}
and that for any complex numbers $\lambda_1,\ldots,\lambda_k$, we
have
\begin{equation}\label{5NormD}
\Bigl(\sum_{i=1}^k\vert\lambda_i\vert^p\Bigr)^{\frac{1}{p}}\,\leq
2 \,\Bignorm{\sum_{i=1}^k \lambda_i\, e_i\otimes e_i\otimes
e_i}_{S_k^p\otimes_\beta \ell^2_k}.
\end{equation}
Now consider
$$
w\,=\,\sum_{i=1}^k e_i\otimes e_i\otimes e_1.
$$
By (\ref{5Haag}) we have
$$
\norm{w}_{S^p_k\{R_k\}_\ell}=\Bignorm{\sum_{i=1}^k e_i\otimes
e_i}_{C_k\otimes_h R_k} = \norm{I_k}_\infty =1.
$$
Recall that if we regard $S^p_k\{R_k\}_\ell$ as the tensor product
$S^p_k\otimes \ell^2_k$, then $e_i\otimes e_j\otimes e_m$
corresponds to $E_{im}\otimes e_j$. Hence we have
\begin{align*}
\bigl(u_1\otimes I_{\ell_k^2}\bigr) (w) & =
k^{-\frac{1}{2p}}\sum_{i=1}^k e_i\otimes e_i\otimes e_i;
\\
\bigl(u_2\otimes I_{\ell_k^2}\bigr) (w) & = k^{-\frac{1}{2p}}
e_1\otimes e_1\otimes e_1;\\
\bigl(u_3\otimes I_{\ell_k^2} \bigr) (w) &= e_1\otimes e_1\otimes
e_1;
\\
\bigl(u_4 \otimes I_{\ell_k^2}\bigr) (w) & =
k^{-\frac{1}{p}}\sum_{i=1}^k  e_i\otimes e_1\otimes e_i.
\end{align*}
Consequently,
$$
P\bigl(u\otimes I_{\ell_k^2}\bigr) (w) =\frac{1}{4}\,\Bigl(
k^{-\frac{1}{2p}}\sum_{i=1}^k e_i\otimes e_i\otimes e_i \,
+\,\bigl(k^{-\frac{1}{2p}} + 1 + k^{-\frac{1}{p}}\bigr) e_1\otimes
e_1\otimes e_1\Bigr).
$$
Applying (\ref{5NormD}) and (\ref{5NormP}), we deduce that
\begin{align*}
\Bigl(\bigl( 2k^{-\frac{1}{2p}} + 1 + k^{-\frac{1}{p}}\bigr)^p\,
+\, (k-1)k^{-\frac{1}{2}}\Bigr)^{\frac{1}{p}}\,& \leq\, 8\bignorm{
P\bigl(u\otimes I_{\ell_k^2}\bigr) (w)}_\beta\\ &\,\leq 8
\bignorm{u\otimes I_{\ell^2_k}\colon
S_k^p\{R_k\}_\ell\longrightarrow S_k^p\otimes_\beta \ell^2_k}.
\end{align*}
This proves (1).
\end{proof}

\end{document}